\documentclass{amsart}

\usepackage[letterpaper, margin=1in]{geometry}

\usepackage{amssymb,amsmath,amsthm,amsfonts}
\usepackage {latexsym}
\usepackage{bbm,enumerate}

\allowdisplaybreaks
\newcommand{\nc}{\newcommand}
\nc{\les}{\lesssim}
\nc{\nit}{\noindent}
\nc{\nn}{\nonumber}
\nc{\D}{\partial}
\nc{\diff}[2]{\frac{d #1}{d #2}}
\nc{\diffn}[3]{\frac{d^{#3} #1}{d {#2}^{#3}}}
\nc{\pdiff}[2]{\frac{\partial #1}{\partial #2}}
\nc{\pdiffn}[3]{\frac{\partial^{#3} #1}{\partial{#2}^{#3}}}
\nc{\abs}[1] {\lvert #1 \rvert}
\nc{\cAc}{{\cal A}_c}
\nc{\cE}{{\cal E}}
\nc{\cF}{{\cal F}}
\nc{\cP}{{\cal P}}
\nc{\cV}{{\cal V}}
\nc{\cQ}{{\cal Q}}
\nc{\cGin}{{\cal G}_{\rm in}}
\nc{\cGout}{{\cal G}_{\rm out}}
\nc{\cO}{{\cal O}}
\nc{\Lav}{{\cal L}_{\rm av}}
\nc{\cL}{{\cal L}}
\nc{\cB}{{\cal B}}
\nc{\cZ}{{\cal Z}}
\nc{\cR}{{\cal R}}
\nc{\cT}{{\cal T}}
\nc{\cY}{{\cal Y}}
\nc{\cX}{{\cal X}}
\nc{\cXT}{{{\cal X}(T)}}
\nc{\cBT}{{{\cal B}(T)}}
\nc{\vD}{{\vec \mathcal{D}}}
\nc{\efield}{\mathcal{E}}
\nc{\vE}{{\vec \efield}}
\nc{\vB}{{\vec \mathcal{B}}}
\nc{\vH}{{\vec \mathcal{H}}}
\nc{\mR}{\mathcal R}
\nc{\cM}{\mathcal M}
\nc{\ty}{{\tilde y}}
\nc{\tu}{{\tilde u}}
\nc{\tV}{{\tilde V}}
\nc{\Pc}{{\bf P_c}}
\nc{\bx}{{\bf x}}
\nc{\bX}{{\bf X}}
\nc{\bXYZ}{{\bf XYZ}}
\nc{\bY}{{\bf Y}}
\nc{\bF}{{\bf F}}
\nc{\bS}{{\bf S}}
\nc{\dV}{{\delta V}}
\nc{\dE}{{\delta E}}
\nc{\TT}{{\Theta}}
\nc{\dPsi}{{\delta\Psi}}
\nc{\order}{{\cal O}}
\nc{\Rout}{R_{\rm out}}
\nc{\eplus}{e_+}
\nc{\eminus}{e_-}
\nc{\epm}{e_\pm}
\nc{\sgn}{\,\textrm{sgn}}
\nc{\eps}{\varepsilon}
\nc{\vnabla}{{\vec\nabla}}
\nc{\G}{\Gamma}
\nc{\w}{\omega}
\nc{\mh}{h}
\nc{\mg}{g}
\nc{\vphi}{\varphi}
\nc{\tlambda}{\tilde\lambda}
\nc{\be}{\begin{equation}}
\nc{\ee}{\end{equation}}
\nc{\ba}{\begin{eqnarray}}
\nc{\ea}{\end{eqnarray}}

\nc{\g}{\gamma}
\nc{\ol}{\overline}

\newtheorem{theorem}{Theorem}[section]
\newtheorem{lemma}[theorem]{Lemma}
\newtheorem{prop}[theorem]{Proposition}
\newtheorem{corollary}[theorem]{Corollary}
\newtheorem{defin}[theorem]{Definition}
\newtheorem{rmk}[theorem]{Remark}

\newtheorem{example}[theorem]{Example}
\nc{\pT}{\partial_T}
\nc{\pz}{\partial_z}
\nc{\pt}{\partial_t}
\nc{\la}{\langle}
\nc{\ra}{\rangle}
\nc{\infint}{\int_{-\infty}^{\infty}}
\nc{\halfwidth}{6.5cm}
\nc{\figwidth}{10cm}
\newcommand{\f}{\frac}

\nc{\nlayers}{L} \nc{\nsectors}{M}
\nc{\indicator}{\mathbf{1}}
\nc{\Rhole}{R_{\rm hole}}
\nc{\Rring}{R_{\rm ring}}
\nc{\neff}{n_{\rm eff}}
\nc{\Frem}{F_{\rm rem}}
\nc{\R}{\mathbb R}
\nc{\Z}{\mathbb Z}
\nc{\C}{\mathbb C}
\nc{\DD}{\Delta}
\nc{\cD}{\mathcal D}
\nc{\lnorm}{\left\|}
\nc{\rnorm}{\right\|}
\nc{\rnormp}{\right\|_{\ell^{p,\eps}}}
\nc{\rar}{\rightarrow} 
\newcommand{\framedtext}[1]{%
	\par%
	\noindent\fbox{%
		\parbox{\dimexpr\linewidth-2\fboxsep-2\fboxrule}{#1}%
	}%
}
\begin{document}

\begin{abstract}

	We investigate $L^1\to L^\infty$ dispersive estimates for the one dimensional Dirac  equation with a potential.  In particular, we show that the Dirac evolution satisfies the natural $t^{-\f12}$ decay rate, which may be improved to $t^{-\f32}$ at the cost of spatial weights when the thresholds are regular. 
	We classify the structure of threshold obstructions, showing that there is at most a one dimensional space at each threshold.  We show that, in the presence of a threshold resonance, the Dirac evolution satisfies the natural decay rate, and satisfies the faster weighted bound except for a piece of rank at most two, one per threshold.  Further, we prove high energy dispersive bounds that are near optimal with respect to the required smoothness of the initial data. To do so we use a variant of a high energy argument that was originally developed to study Kato smoothing estimates for magnetic Schr\"odinger operators. This method has never been used before to obtain $L^1 \to L^\infty$ estimates. As a consequence of our analysis we prove a uniform limiting absorption principle, Strichartz estimates and prove the existence of an eigenvalue free region for the one dimensional Dirac operator with a non-self-adjoint potential.
	
\end{abstract}

\title[Dispersive Estimates for  the Dirac Equation]{\textit{On the one dimensional Dirac Equation with potential}}

\author[M.~B. Erdo\smash{\u{g}}an, W.~R. Green]{M. Burak Erdo\smash{\u{g}}an and William~R. Green}
\thanks{The first author is supported by Simons Foundation Grant 634269.
	The second  author is supported by Simons Foundation Grant 511825.}  
\address{Department of Mathematics \\
University of Illinois \\
Urbana, IL 61801, U.S.A.}
\email{berdogan@illinois.edu}
\address{Department of Mathematics\\
Rose-Hulman Institute of Technology \\
Terre Haute, IN 47803, U.S.A.}
\email{green@rose-hulman.edu}
%\subjclass{35Q41, 42B20}

\framedtext{
	This version corrects an error in Lemma~2.2 of the published version which also necessitates changes to the statement of Theorems~1.1 and 2.3 for high energies only.	
	The authors thank Joseph Kraisler, Amir Sagiv and Michael Weinstein for pointing out the error in  the published version.  July 18, 2023.}\\

\maketitle

\section{Introduction}

We consider the linear Dirac equations with potential, in one spatial dimension
\begin{align}\label{eqn:Dirac}
i\partial_t \psi(x,t)=(D_m+V(x))\psi(x,t), \qquad
\psi(x,0)=\psi_0(x).
\end{align}
Here $\psi(x,t) \in \mathbb C^{2 }$ when the spatial
variable $x\in \mathbb R$.  The free Dirac operator
$D_m$ is defined by
\begin{align}\label{eqn:Dmdef}
D_m=i\alpha \partial_x +\beta m= i\begin{pmatrix}
-1 & 0\\ 0 & 1
\end{pmatrix}\partial_x +m\begin{pmatrix}0 & 1\\ 1 & 0 \end{pmatrix}
\end{align}
where $m>0$ is a constant, and 
the Hermitian matrices $\alpha, \beta$ satisfy the anti-commutation relationships
\be\label{eqn:anticomm}
\alpha \beta+\beta \alpha= 
\mathbbm O_{\mathbb C^{2}},\;\;\;\;\;
\alpha^2=\beta^2 = \mathbbm 1_{\mathbb C^{2 }}.
\ee
Physically, the Dirac equation connects the theories of quantum mechanics and special relativity to describe the evolution of quantum particles moving at near luminal speeds.  The model is first order in time to allow for a wave function interpretation of the solution ``spinor" $\psi(x,t)$ while allowing for a Pythagorean energy addition relation, $E^2=(cp)^2+(mc)^2$ where $E$ is the energy of the particle, $m$ is the mass, $c$ is the speed of light and $p$ is the momentum, required by relativistic models. Futher, the Dirac equation allows for the influence of an external potential in a manner that is relativistically invariant.  For a more thorough introduction to the Dirac equation see \cite{Thaller}.

The Dirac equation may be viewed as a square root of a system of
Klein-Gordon, equations.  This motivates the
following relationship, which follows from the relationships in \eqref{eqn:anticomm}
\begin{align}
(D_m-\omega)(D_m+\omega)=(i\alpha \partial_x +\beta m-\omega)
(i\alpha \partial_x +\beta m+\omega)=(-\partial_{xx}+m^2-\omega^2)\mathbbm 1_{\mathbb C^{2 }} .
\end{align}
This allows us to formally define the free Dirac resolvent
operator $\mathcal R_0(z)=(D_m-z)^{-1}$ in terms of the
free resolvent $R_0(z)=(-\Delta-z)^{-1}$ for the Schr\"odinger operator for $z$ in the resolvent set.
That is,
\begin{align}\label{eq:resolv id}
\mathcal R_0(z)=(D_m+z) R_0(z^2-m^2), \,\,\,z\in \C\setminus \sigma(D_m), 
\end{align}
where $\sigma(D_m)$ is the spectrum of the free operator which is purely absolutely continuous and unbounded both above and below:
$$
\sigma(D_m)=(-\infty,-m]\cup [m,\infty).
$$
In one physical interpretation, the Dirac equation couples the evolution of massive particle and anti-particles.
For suitable potential functions $V$, one has a Weyl criterion for $H:=D_m+V$.  That is, $\sigma_{ac}(H)=\sigma_{ac}(D_m)=(-\infty,-m]\cup [m,\infty)$.  For a rather general class of potentials including the ones we consider here, there are no eigenvalues embedded in the continuous spectrum, \cite{BC1}.  

Throughout the paper we use the notation $\la x\ra:=(1+|x|^2)^{\frac12}$.  When we write $|V(x)|\les \la x \ra^{-\delta}$ we mean that each entry $V_{ij}(x)$ of the potential matrix satisfies the bound $|V_{ij}(x)|\les \la x\ra^{-\delta}$.  Similarly, if we write $V\in L^p$ we mean each entry of $V$ is in $L^p$.  The weighted $L^p$ spaces $L^{p,\sigma}$ are defined by $\{f: \la \cdot\ra^{\sigma}f \in L^p\}$.  Many of these spaces are used for $\mathbb C^2$-valued functions, which will be clear from context.   We write $P_{ac}(H)$ to denote projection onto the absolutely continuous spectral subspace of $L^2$ associated to the Dirac operator $H$.  Finally, we write $a-$ to mean $a-\epsilon$ for an arbitrarily small, but fixed $\epsilon>0$.  Similarly, we write $a+$ to mean $a+\epsilon$.  

We say that the threshold energies are regular if there are no distributional solutions to $H\psi= m\psi$ or $H\psi=-m\psi$ for $\psi \in L^\infty$.  This may also be characterized by the uniform boundedness of the resolvent operator $(H-\lambda)^{-1}$ as $\lambda \to m$ (resp. $-m$) between certain weighted $L^2$ spaces.  We provide a detailed characterization of the threshold obstructions in Section~\ref{sec:spec}.

Our first result is    
\begin{theorem}\label{thm:main reg}
	
	Assume that $V$ is self adjoint and  $|V(x)|\les \la x\ra^{-\delta}$.  If the threshold energies $\pm m$ are regular, then
	\begin{enumerate}[i)]
		\item if $\delta>3$, and $|\partial_x V(x)|\les \la x\ra^{-1-}$ we have the dispersive bounds
		\begin{align*}
		&\|e^{-itH}P_{ac}(H)\la H\ra^{-\frac32-}\|_{L^1\to L^\infty} \les  \la t\ra^{-\f12},\\
		&\|e^{-itH}P_{ac}(H)\la H\ra^{-1-}\|_{L^1\to L^\infty} \les  1.
		\end{align*}
		\item if $\delta>5$, and $|\partial_x V(x)|\les \la x\ra^{-2-}$, then for $|t|>1$ we have 
		$$\|e^{-itH}P_{ac}(H)\la H\ra^{-\frac32-}\|_{L^{1,\tau}\to L^{\infty,-\tau}} \les  
		|t|^{-\f12-\tau}, \,\,0\leq \tau\leq 1.$$
	\end{enumerate}
\end{theorem}

We note that the continuity and differentiability of the potential is required only for the high energy argument. We note that  the $|t|^{-\f12}$ dispersive decay bounds for the free Dirac and Klein-Gordon equations require three halves derivative loss.  This suggests that the bounds stated above are essentially sharp with respect to the required differentiability, given by the negative powers of $\la H\ra$, of the initial data.    We also note that the weighted bounds are integrable in time when $\tau>\f12$.  The high energy contribution to the dispersive bound provides some difficulties.  Namely, both free and the perturbed resolvent operators and their derivatives don't decay in the spectral variable $z$ as $|z|\to \infty$.  This limits the bootstrapping argument of Agmon, \cite{agmon}, to compact subsets of the continuous spectrum.  This typically requires one to assume more smoothness on the initial data to counteract this lack of decay and close the large energy argument.  To overcome this obstacle and prove sharp estimates, we adapt the high energy method used in \cite{EGS1,EGS2,EGG} to the case of the one dimensional Dirac operator. This method has never been used before to obtain $L^1\to L^\infty$ dispersive estimates.  In addition to proving the dispersive bounds, this allows us to prove a uniform limiting absorption principle, establish a family of Strichartz estimates and establish results about the spectrum of one dimensional Dirac operators, see Theorem~\ref{cor:uniform LAP}, and Corollaries~\ref{cor:Strichartz} and \ref{cor:eval lack} below.

We also provide a classification of the effect of threshold resonances on the dynamics of the solution.  The existence of threshold resonances only effects the low energy evolution, the high energy bounds are unaffected.  Namely, with $\chi$ a smooth cut-off to a sufficiently small neighborhood of the threshold, and for $\C^2$-valued $f,g$ denoting $\la f,g\ra= \int f(x)g^*(x)\, dx$, we have the following low energy bounds.

\begin{theorem}\label{thm:main nonreg}
	
	Assume that $V$ is self adjoint and  $|V (x)|\les \la x\ra^{-\delta}$.  If one or both of the threshold energies $\pm m$ are not regular, then
	\begin{enumerate}[i)]
		\item If $\delta>5$,  we have the low energy dispersive bound
		$$
		\|e^{-itH}P_{ac}(H)\chi(H) \|_{L^1\to L^\infty} \les \la t\ra^{-\f12}.
		$$
		\item If $\delta>9$,  then for $|t|>1$ there exists an operator $F_t$ of rank at most two (one for each threshold) satisfying 
		$\|  F_t\|_{L^1\to L^\infty}\les | t|^{-\f12}$, 
		so that 
		$$
		\|e^{-itH}P_{ac}(H)\chi(H)- F_t\|_{L^{1,\tau}\to L^{\infty,-\tau}} \les |t|^{-\f12-\tau},\;\;\;0\leq \tau\leq 1.
		$$
		\item Furthermore, for $\delta>9$ and $|t|>1$, with stronger weights we can express the operator $F_t$ from the previous statement as a rank two projection to canonical resonance functions (one for each threshold) for $f\in L^1$ as follows: 
		%				$$
		%			\|e^{-itH}P_{ac}(H)\chi(H)-F_t \|_{L^{1,2\tau}\to L^{\infty,-2\tau}} \les |t|^{-\f12-\tau},\;\;\;0\leq \tau \leq 1.
		%		$$
		%		Here, 	  for $f\in L^1$,
		$$
		F_tf=t^{-\frac12} \big( c_+e^{-imt} \psi_+ \la f, \psi_+ \ra +c_- e^{imt} \psi_- \la f, \psi_-  \ra \big) +O(|t|^{-\f12-\tau} \la x\ra^{2\tau} \| f \|_{L^{1,2\tau}}),
		$$
		where $\psi_\pm\in L^\infty$ are the canonical resonance functions, that is distributional solutions to $H\psi_{\pm}=\pm m\psi_\pm$, and the constants $c_\pm$ can be computed explicitly; for $c_+$, see Proposition~\ref{prop:Ft real}.
		
	\end{enumerate}
	
\end{theorem}

%Effectively, this shows t
The effect of a threshold resonance is to produce a slower decaying portion of the evolution that is rank at most one for a resonance at the positive  ($\lambda=m$)  or negative ($\lambda=-m$) thresholds respectively.  This theorem may be combined with the high energy argument in Section~\ref{sec:high} to provide a dispersive bound without the low energy cut-off as the effect of a threshold obstruction only affects an arbitrarily small neighborhood of $\lambda=m$ ($\lambda=-m$ respectively).  The rank at most two operator we construct, as noted in the final statement of Theorem~\ref{thm:main nonreg} is a time-dependent scalar function multiplying the projection onto the one dimensional space of resonances at each threshold, at the cost of greater spatial weights.  See Propositions~\ref{prop:nonreg} and \ref{prop:Ft real} below for a construction of this operator.

The decay requirement on the potential for these results is not necessarily optimal.  The different values of $\delta$ in Theorems~\ref{thm:main reg} and \ref{thm:main nonreg} are required to develop appropriate expansions for the resolvent around the threshold energies. 

%%%%%%%%%%%%%%%%%%%%%%%%%
%%%%%%%%%%%%%%%%%%%%%%%%%
%%%%%%%%%%%%%%%%%%%%%%%%%
%%%%%%%%%%%%%%%%%%%%%%%%%
%%%%%%%%%%%%%%%%%%%%%%%%%
%%%%%%%%%%%%%%%%%%%%%%%%%
%%%%%%%%%%%%%%%%%%%%%%%%%
%%%%%%%%%%%%%%%%%%%%%%%%%
%%%%%%%%%%%%%%%%%%%%%%%%%
%%%%%%%%%%%%%%%%%%%%%%%%%
%%%%%%%%%%%%%%%%%%%%%%%%%
%%%%%%%%%%%%%%%%%%%%%%%%%
%%%%%%%%%%%%%%%%%%%%%%%%%
%%%%%%%%%%%%%%%%%%%%%%%%%
%%%%%%%%%%%%%%%%%%%%%%%%%

As in the multi-dimensional results, \cite{egd,EGT2d,EGT,EGG}, our techniques also provide useful insights on the spectral theory of the perturbed operator.  Our resolvent expansions obtained below also show that $\mR_V(\lambda)=(H-\lambda)^{-1}$ is a uniformly bounded operator between weighted $L^2$ spaces in a neighborhood of $\lambda =\pm m$ if the thresholds are regular.\footnote{In the case of a resonance, using our bounds one may conclude that $(\lambda \mp m)\mR_V(\lambda)$ is uniformly bounded in a neighborhood of the threshold. } This implies a limiting absorption principle bound in a neighborhood of each threshold, as well as showing the absence of eigenvalues in these neighborhoods.  As a consequence, there are only finitely many eigenvalues in the spectral gap $(-m,m)$.  Combining this with the high energy argument, see Lemma~\ref{lem:LAP} below, we obtain a uniform limiting absorption principle over the spectrum, namely
\begin{theorem}\label{cor:uniform LAP}
	
	Under the assumption that the threshold energies are regular, and if $V$ has continuous entries satisfying  $|V(x)| \les \la x\ra^{-3-}$, we have the uniform resolvent bounds:
	$$
	\sup_{|\lambda| >m} \ 
	\| \la x\ra^{-\sigma}   (H - (\lambda + i0))^{-1}
	\la x\ra^{-\sigma}\|_{2\to2} \les 1, \quad \text{provided} \quad \sigma>1. 
	$$
	
\end{theorem}

There are several immediate consequences of this result:
\begin{corollary}\label{cor:Strichartz}
	Let  $V$ be a  self-adjoint matrix, with continuous entries satisfying  
	$|V(x)|  \les  \la x\ra^{-3- } $. If the threshold energies are regular, we have 
	\begin{equation}\label{eq:strichmassive} 
	\|\la \nabla\ra^{-\theta} e^{-itH} P_{ac}(H) f\|_{L_t^p(L_x^q)}
	\les \|f\|_{L^2(\R)} 
	\end{equation}
	provided that $2\leq q,r\leq \infty$ with
	\begin{equation*}
	\theta \geq \frac12 + \frac1p - \frac1q \quad \text{and} \quad 
	\frac{2}{p}+\frac{1}{q}=\frac{1}{2}.
	\end{equation*}
	
\end{corollary}

Finally, the argument used for high energies in Theorem~\ref{cor:uniform LAP} requires only that $|V(x)|\les \la x\ra^{-1-}$ with continuous entries, in particular it doesn't require $V$ to be self-adjoint.  In addition, $\sigma>\f12$ suffices. The operator $D_m$ is self-adjoint, $H$ has the same domain as $D_m$ and for unit functions $\eta$ in the domain the quadratic form $\la H\eta, \eta\ra$ is confined to a strip of finite width around the real axis.  
Consequently, any $\lambda\in \R$ with $|\lambda |$ sufficiently large cannot be an embedded eigenvalue or resonance.  The perturbation argument in \cite{EGG} shows that the eigenvalue-free zone
extends to a sector of the complex plane  containing a portion of the real line sufficiently far from zero energy.  

\begin{corollary}\label{cor:eval lack}
	
	Let $V$ be any matrix satisfying $|V(x)|\les \la x\ra^{-1-}$ with continuous entries.  Then, there exists a $m<\lambda_1<\infty$ and a $\delta>0$ depending on $V$, $m$ and $\sigma>\f12$ so that
	$$
	\sup_{\substack{|\lambda| > \lambda_1\\ 0 < |\gamma| < \delta |\lambda|}} 
	\| \la x\ra^{-\sigma} (H - (\lambda + i\gamma))^{-1} \la x\ra^{-\sigma}\|_{2\to 2} \les 1.
	$$
	As a result, there is a compact subset of the complex plane outside of which the spectrum of $H$ is confined to the real axis.
	
\end{corollary}

%%%%%%%%%%%%%%%%%%%%%%
%%%%%%%%%%%%%%%%%%%%%%
%%%%%%%%%%%%%%%%%%%%%%%
%%%%%%%%%%%%%%%%%%%%%%
%%%%%%%%%%%%%%%%%%%%%%%
%%%%%%%%%%%%%%%%%%%%%%
%%%%%%%%%%%%%%%%%%%%%%

The study of dispersive estimates and the effect of threshold obstructions for the Dirac equation have only more recently been studied compared to other dispersive equations such as the Schr\"odinger, wave and Klein-Gordon equation.
To the authors' knowledge, this is the first study of uniform, $L^\infty$ based, dispersive estimates for the one dimensional equation.  Estimates for the one dimensional operator on weighted $L^2$ spaces when the thresholds are regular were obtained by Kopylova in \cite{Kopy},  Strichartz and Mizumachi estimates were obtained by Pelinovski and Stefanov in \cite{PS} for an exponentially decaying potential in service of studying the stability of solitons to a non-linear Dirac equation.  We note that the class of potentials we consider include those that arise naturally when linearizing around soliton solutions of non-linear Dirac equations.

The dispersive estimates for the three dimensional Dirac equation is more studied going back to the work of Boussa\"id \cite{Bouss1}, and D'Ancona and Fanelli, \cite{DF} in the massive $m>0$ case.  Earlier results on the free Dirac operator were obtained in \cite{BB}, while the analysis in \cite{BGW} used a careful study of the Jost functions.  The characterization of threshold obstructions as resonances and eigenvalues along with their effect on the dispersive bounds in three dimensions has  been studied by the authors and Toprak, \cite{EGT}.  Dispersive bounds for two dimensional Dirac has been studied by the authors, \cite{egd}, with Toprak \cite{EGT2d} in the massive case and with Goldberg in the massless case \cite{EGG2}.    Much of the work somehow relies on the techniques developed in the study of other dispersive equations, notably the Schr\"odinger equation \cite{Mur,GS,ES,KS,GolT,Miz,ebru,EKMF,Hill}, which analyze the effect of threshold energy obstructions. 
 
Our results in one dimension are inspired by previous work on the Schr\"odinger equation.  In \cite{GS} Goldberg and Schlag used an analysis based on the Jost functions to  prove a  $|t|^{-\f12}$ decay rate for the Schr\"odinger operator whether zero energy was regular or not.  In the survey paper \cite{Scs} Schlag also showed that if zero energy is regular one can obtain a faster $|t|^{-\f32}$ bound at the cost of spatial weights, this was motivated by results for a   matrix equation arising in linearization about special solutions for a nonlinear equation in work with Krieger, \cite{KS}.  The assumptions on the potential and spatial weights required has been lessened in subsequent works, \cite{GolT,Miz,EKMF,Hill}.  The sharpest results were obtained in \cite{EKMF}, and in Hill's Ph.D. thesis, \cite{Hill}.  These results are analogous to what we prove in Theorem~\ref{thm:main reg} and in the first claim of Theorem~\ref{thm:main nonreg}.  We note that statement of the form found in the second claim of Theorem~\ref{thm:main nonreg}, which requires smaller spatial weights, have not been obtained for the one dimensional Schr\"odinger operator.  Goldberg's work \cite{GolT} proves a statement of form of the third claim in Theorem~\ref{thm:main nonreg}.  Using our methods these results can be obtained with a rank-one operator $F_t$ for the Schr\"odinger equation.  A similar result for the two dimensional Schr\"odinger operator with an s-wave resonance only at the threshold was obtained by Toprak in \cite{ebru}.
In our analysis, we eschew the approach of using the Jost functions and instead provide a careful analysis of the spectral measure by studying the resolvent operators, as in the multi-dimensional cases \cite{egd,EGT,EGT2d,EGG2}.

There is also much interest in the study of non-linear Dirac equations.  See for example, \cite{EV,PS,BH3,BH,CTS,BC2} and the recent monograph by Boussa\"id and Comech \cite{BCbook}.  There is a longer history in the study of spectral properties of Dirac operators.  Limiting absorption principles for the Dirac operators have been studied in \cite{Yam,GM,EGG,CGetal}.    The lack of embedded eigenvalues, singular continuous spectrum and other spectral properties is well established, \cite{BG1,GM,MY,CGetal,BC1}.  In particular, for the class of self-adjoint potentials we consider, the Weyl criterion implies that $\sigma_{ac}(H)=\sigma(D_m)$ and the remainder of the spectrum is composed of eigenvalues confined to the spectral gap $\sigma_p(H)\subseteq [-m,m]$.  In general, there need not be finitely many eigenvalues in the gap \cite{Thaller}.

Since $H$ is self-adjoint, the functional calculus allows us to represent the solution as an integral over the spectrum via the Stone's formula:
\begin{align}\label{eqn:stone}
	e^{-itH}P_{ac}(H)=\frac{1}{2\pi i} \int_{\sigma_{ac}(H)} e^{-it\lambda}[\mR_V^+-\mR_V^-](\lambda)\, d\lambda.
\end{align}
Here $\mR_V^{\pm}(\lambda)$ are the limiting perturbed Dirac resolvents defined by
$$
	\mR_V^\pm(\lambda)=\lim_{\epsilon \searrow 0} (D_m+V-(\lambda \pm i \epsilon))^{-1}, \qquad \lambda \in \sigma_{ac}(H).
$$
Their difference provides the spectral measure.

%%%%%%%%%%%%%%%%%%%%%%%%%
%%%%%%%%%%%%%%%%%%%%%%%%%
%%%%%%%%%%%%%%%%%%%%%%%%%
%%%%%%%%%%%%%%%%%%%%%%%%%
%%%%%%%%%%%%%%%%%%%%%%%%%
%%%%%%%%%%%%%%%%%%%%%%%%%
%%%%%%%%%%%%%%%%%%%%%%%%%
%%%%%%%%%%%%%%%%%%%%%%%%%
%%%%%%%%%%%%%%%%%%%%%%%%%
%%%%%%%%%%%%%%%%%%%%%%%%%
%%%%%%%%%%%%%%%%%%%%%%%%%
%%%%%%%%%%%%%%%%%%%%%%%%%
%%%%%%%%%%%%%%%%%%%%%%%%%
%%%%%%%%%%%%%%%%%%%%%%%%%
%%%%%%%%%%%%%%%%%%%%%%%%%

The paper is organized as follows. In Section~\ref{sec:free} we study the free Dirac evolution, and develop oscillatory integral bounds that are used throughout the paper.  In Section~\ref{sec:Minv}, we develop expansions for the perturbed resolvent in a neighborhood of the threshold energy $\lambda=m$ whether the threshold is regular or not.  These expansions may then be used to build an appropriate spectral measure to analyze the perturbed evolution.  We then utilize these expansions in Section~\ref{sec:low disp} to prove low energy dispersive bounds.  In Section~\ref{sec:spec} we characterize the threshold obstructions in terms of distributional solutions to $H\psi=\pm m\psi$ and characterize the spectral subspace of $L^2$ the obstructions induce.  In Section~\ref{sec:LAP} we show how the resolvent expansions used to obtain the dispersive bounds may be adapted to prove the uniform limiting absorption principle in Theorem~\ref{cor:uniform LAP}.
Finally, in Section~\ref{sec:high} we prove high energy dispersive bounds when the spectral parameter is bounded away from the threshold energies.

\section{The free evolution}\label{sec:free}

We begin by looking at the free Dirac evolution.  Using the Stone's formula, \eqref{eqn:stone}, we may write
\begin{equation}\label{free1}
	e^{-itD_m}(x,y)=\frac{1}{2\pi i} \int_{\sigma_{ac} (D_m)} e^{-it\lambda} [\mR_0^+-\mR_0^-](\lambda)(x,y)\, d\lambda.
\end{equation}
Without loss of generality, throughout the paper we consider the positive branch of the spectrum $[m,\infty)$.  The results for the negative branch $(-\infty, -m]$ follow with minimal changes, see Remark~\ref{rmk:neg branch} below.  From now on whenever we write $D_m$ we mean $D_m\chi_{(m,\infty)}(D_m)$.  Our analysis relies upon reducing the operator bounds to oscillatory integral estimates.  To that end, recall the Van der Corput lemma, \cite{Stein}.
\begin{lemma}\label{lem:vdc1}
	Let  $ \phi $ be a smooth real valued function on $\R$  and $\psi $ be a smooth compactly supported function. If $|\partial_{zz} \phi(z)|\geq \lambda>0$ in   the support of $\psi$, then
	$$
	\bigg|\int_\R e^{ i\phi(z)} \psi(z)  \, dz \bigg|\leq C  \lambda^{-\frac12} \|\partial_z\psi\|_{L^1}.
	$$
	Here $C$ is an absolute constant.
	
\end{lemma}
We  utilize the following implication
of Van der Corput lemma repeatedly. 
\begin{lemma}\label{lem:vdc}
	Let $\psi_j$ be a smooth function supported on the set $|z|\approx 2^j$ when $j>0$ and supported in a small neighborhood of zero when $j=0$.
	Then  
	\begin{multline*}
	\bigg|\int_\R e^{-it\sqrt{z^2+m^2}+izr} \psi_j(z) \, dz \bigg|\\
	\leq C_m
	\min\Big(\|\psi_j\|_{L^1},|t|^{-\f12}2^{\frac{3j}2}\|\partial_z\psi_j\|_{L^1}, |t|^{-\f32}2^{\frac{3j}2} \big\| [\partial_{zz}+ir\partial_z]\big(\frac{\psi_j}z\sqrt{z^2+m^2}\big)\big\|_{L^1}\Big)
	\end{multline*}
	for any $r\in \R$.
	
\end{lemma}

\begin{proof}
	The first bound is obvious. The second bound follows immediately from Lemma~\ref{lem:vdc1}  noting that $|\partial_{zz} (- t\sqrt{z^2+m^2}+ zr)|\gtrsim |t|2^{-3j}$ in the support of $\psi_j$. 
	
	For the last bound, we integrate by parts once using $\partial_z e^{-it\sqrt{z^2+m^2}}=e^{-it\sqrt{z^2+m^2}} \frac{-it z}{\sqrt{z^2+m^2}}$, which leads to 
	$$
	\frac{1}{ it}  \int_\R e^{-it\sqrt{z^2+m^2} } \partial_z \bigg[e^{izr} \frac{\sqrt{z^2+m^2}}{z}   \psi_j(z)\bigg]\, dz 
	= \frac{1}{ it}  \int_\R e^{-it\sqrt{z^2+m^2} +izr}  [ir+\partial_z ] \bigg[\frac{\sqrt{z^2+m^2}}{z}  \psi_j (z)\bigg]\, dz. 
	$$
	From this the bound follows immediately from Lemma~\ref{lem:vdc1} as above.
\end{proof}

By \eqref{eq:resolv id}, and the standard representation of the free Schr\"odinger resolvent kernel using Bessel functions, we have the representation
$$
	\mR_0^\pm (\lambda)(x,y)=(D_m+\lambda)\frac{\pm i e^{\pm i \sqrt{\lambda^2-m^2} |x-y|}}{2\sqrt{\lambda^2-m^2}}.
$$
With the change of variable $\lambda=\sqrt{m^2+z^2}$ (and renaming $\mR_0^\pm (\lambda)$ by  $\mR_0^\pm (z)$), we write 
\begin{align}\label{eqn:R0}
	\mR_0^\pm (z)(x,y)=\frac{\pm i  }{2z}[\mp \alpha   z \textrm{sgn}(x-y) +\beta m+\sqrt{z^2+m^2}I]  e^{\pm i z |x-y|}.
\end{align}
Using the same change of variable in \eqref{free1} and then using $\mR_0^+(z)= \mR_0^-(-z)$, we have
\begin{multline}\label{eq:freeevolv}
	e^{-itD_m}(x,y)=\frac{1}{2\pi i} \int_0^\infty e^{-it\sqrt{z^2+m^2}} \frac{z}{\sqrt{z^2+m^2}}  [\mR_0^+-\mR_0^-](z)(x,y) \, dz\\	=
	\frac{1}{2\pi i} \int_\R e^{-it\sqrt{z^2+m^2}} \frac{z}{\sqrt{z^2+m^2}}   \mR_0^+(z)(x,y) \, dz \\
	=\frac{1}{ 4\pi  } \int_\R e^{-it\sqrt{z^2+m^2}+iz|x-y|} \frac{z}{\sqrt{z^2+m^2}} \big[- \alpha     \textrm{sgn}(x-y) + \frac{1  }{ z}(  m \beta +\sqrt{z^2+m^2}I)\big]   \, dz.
\end{multline}
Here the integrals are understood in the principal value sense.  We have the following bound for the free evolution.
\begin{theorem}\label{thm:giggles}
	
	Let $\chi_j(z)$ be a smooth, even cut-off for the set $ |z |\approx 2^j$.  Then the kernel of the free Dirac evolution satisfies the bound
	$$
		\|e^{-itD_m} \chi_j(D_m)\|_{  L^\infty} \les \min(2^j, |t|^{-\f12}2^{\frac{3j}{2}},|t|^{-\frac32} \la x-y\ra2^{\frac{3j}{2}} ).
	$$
	Take $\chi_0(z)$ to be a smooth cut-off for a sufficiently small neighborhood of the threshold energy $z=0$ ($\lambda=m$)  we have
	$$
		\|e^{-itD_m} \chi_0(D_m)\|_{ L^\infty} \les \la t \ra^{-\f12}.
	$$
	Furthermore,  for $|t|>1$ we have the weighted estimate 
	$$
		\|e^{-itD_m} \chi_0(D_m)-F^0_t\|_{L^\infty} \les |t|^{-\f32 }  \la x-y\ra,
	$$
	where 
	\be\label{Ft0def}
	F^0_t(x,y)= \frac{m}{4\pi }(\beta+I) \int_\R e^{-i  t\sqrt{z^2+m^2}+ iz|x-y|}   \frac{\chi_0(z)}{\sqrt{z^2+m^2}}    \, dz.
	\ee
\end{theorem}
\begin{proof}
First consider the energies away from zero. Let 
$$
\psi_j(z,x,y)=\chi_j(z) \frac{z}{\sqrt{z^2+m^2}} \big[- \alpha     \textrm{sgn}(x-y) + \frac{1  }{ z}(  m \beta +\sqrt{z^2+m^2}I)\big].
$$
Note that $\|\psi_j\|_{L^1}\les 2^{j}$ and $\|\partial_z\psi_j\|_{L^1}\les 1$ uniformly in $x,y$. Therefore, using Lemma~\ref{lem:vdc}, we can bound the right hand side of \eqref{eq:freeevolv} by 
$ \min(2^j, |t|^{-\f12}2^{\frac{3j}{2}})$.  To obtain the weighted bound, we again apply Lemma~\ref{lem:vdc} (with $r=|x-y|$) noting that
$$\big\| [\partial_{zz}+ir\partial_z]\big(\frac{\psi_j}z\sqrt{z^2+m^2}\big)\big\|_{L^1}
=\big\| [\partial_{zz}+ir\partial_z]\big( - \alpha     \textrm{sgn}(x-y) + \frac{1  }{ z}(  m \beta +\sqrt{z^2+m^2}I)\big)\chi_j(z)\big\|_{L^1}\les \la r\ra.$$

The first  bound  for the low energies follows similarly, noting that 
$$
\psi_0(z,x,y):=\chi_0(z) \frac{z}{\sqrt{z^2+m^2}} \big[- \alpha     \textrm{sgn}(x-y) + \frac{1  }{ z}(  m \beta +\sqrt{z^2+m^2}I)\big] 
$$
satisfies $\|\psi \|_{L^1_z}, \|\partial_z \psi \|_{L^1_z} \les 1$. For the weighted bound we write 
$$
\psi_0(z,x,y)= \frac{\chi_0(z)}{\sqrt{z^2+m^2}} m(\beta+I)+ \chi_0(z) \frac{z}{\sqrt{z^2+m^2}} \widetilde{\psi}_0(z,x,y),
$$
 where the error term satisfy $|\widetilde\psi_0|, |\partial_z\widetilde\psi_0|, |\partial_{zz}\widetilde\psi_0| \les 1$ in the support of $\chi_0$. This leads to the weighted bound for the contribution of $\widetilde\psi_0$ as in the case $j> 0$, and the first term yields the operator $F^0_t$.
\end{proof}

\section{Resolvent expansions near the threshold}\label{sec:Minv}

To understand the contribution of the low energy portion of the evolution to the Stone's formula, we need to understand the behavior of the integral kernel of the perturbed resolvent operators $\mR_V^\pm(\lambda)(x,y)$ as $\lambda \to m^+$.  As in the analysis of the free evolution in Theorem~\ref{thm:giggles}, we utilize the change of variable $\lambda=\sqrt{m^2+z^2}$ and rename $\mR_V^\pm (\lambda)$ as $\mR_V^\pm (z)$.
Under the assumption that the matrix 
$V:\mathbb R  \to \mathbb C^{2}$ is self-adjoint,
the spectral theorem allows us to write
$$
V=B^*\left(\begin{array}{cc}
\lambda_1 & 0 \\ 0 &\lambda_2
\end{array}\right)B ,
$$
with $\lambda_j \in \mathbb R$ and $B $ unitary.  We can
further write $\eta_j =|\lambda_j|^{\f12}$,
\begin{align*}
V=B^*\left(\begin{array}{cc}
\eta_1 & 0 \\ 0 & \eta_2
\end{array}\right) U \left(\begin{array}{cc}
\eta_1 & 0 \\ 0 & \eta_2
\end{array}\right)B, \qquad \textrm{with} \qquad
U=\left(\begin{array}{cc}
\textrm{sgn}(\lambda_1) & 0 \\ 0 & \textrm{sgn}(\lambda_2)
\end{array}
\right).
\end{align*}
So that, with
$$
v=\left(\begin{array}{cc}
\eta_1 & 0 \\ 0 & \eta_2
\end{array}\right)B,
$$
we can write $V=v^*Uv$.  This allows us to employ the
symmetric resolvent identity:
\begin{align}\label{eqn:symm res id}
\mathcal R_V^\pm(z)=\mR_0^\pm(z) -\mR_0^\pm(z) v^*[M^{\pm}(z)]^{-1}v\mR_0^\pm(z),
\end{align}
where
\begin{align}\label{eqn:M defn}
	M^\pm(z)=U+v\mR_0^\pm(z) v^*.
\end{align}
We now seek to invert $M^\pm(z)$ 
in a neighborhood of $z=0$.  We only consider the ``$+$'' case and drop the superscript ``$+$''.  We write
$$
v=\begin{pmatrix}
a & b \\ c & d
\end{pmatrix},  \qquad v^*=\begin{pmatrix} \overline a & \overline c \\ \overline b & \overline d \end{pmatrix}.
$$
For the convenience of the reader we outline some notation that is utilized throughout the remainder of the paper.  We say that an operator $T:L^2\to L^2$ with integral kernel $T(x,y)$ is absolutely bounded if the operator with integral kernel $|T(x,y)|$ is also a bounded operator on $L^2$.  Finite rank and Hilbert-Schmidt operators are absolutely bounded.

To control the size of an absolutely bounded operator with respect to the spectral variable $z$ we write $\Gamma_{\theta}^k$ to denote an absolutely bounded operator that satisfies the bound
$$
	\sum_{j=0}^k |z|^j \big\| | \partial_z^j \Gamma_{\theta}^k| \big\|_{L^2\to L^2}\les z^\theta, \quad 0<|z|<z_0,
$$
for some $z_0>0$. 
We will utilize this notation for $k=0,1,2$.  Similarly, we denote a $z$ independent, absolutely bounded operator as $\Gamma$.  Furthermore, we denote constants whose exact values are not important for our analysis by $c_j$. We note that the operators written with this notation and constants are allowed to vary from line to line.     Finally, we write $f(z)=O_k(z^\ell)$ to denote that $|\partial_z^j f(z)|\les z^{\ell-j}$ for $0\leq j\leq k$ and $0<|z|<z_0$.

To invert $M(z)$ near $z=0$, we need to  develop appropriate expansions for $\mR_0(z)$.
Noting the representation for the free resolvent in \eqref{eqn:R0}, and expanding $\sqrt{m^2+z^2}$ near $z=0$ we see that
 \begin{multline}
 \label{resolventex}
      \mathcal{R}_0 (z)(x,y)= \big[  i \alpha \partial_x   + m \beta + \sqrt{m^2+ z^2} I \big] \frac{ie^{iz|x-y|}}{2z} = \\ 
      \frac{i}2\big[ -  \alpha\, \textrm{sgn}(x-y) + \frac{m}{z}(\beta+I)+ \frac{z }{2m} I + O_2(z^3) I \big] e^{iz|x-y|}. 
  \end{multline}
 We have 
  \begin{lemma} \label{lem:R0exp} Let $r:= |x-y|$,    $0 < z <1$. We have the following expansions for the free resolvent
   \begin{align} \label{eq:R0exp-1}\mathcal{R}_0(z)(x,y) &=\frac{im}{2z}(\beta+I)+O(1),\\
   \label{eq:R0exp_00}
  &= \frac{im}{2z}(\beta+I)+G_0 (x,y)+O_1\big(z^\ell \la r\ra^{1+\ell}\big),\,\,0\leq \ell \leq 1,\\   
 \label{eq:R0exp_0}   &= \frac{im}{2z}(\beta+I)+G_0 (x,y) +  z G_1 +  O_2 \big(z^{1+\ell}  \la r\ra^{2+\ell}  \big), \,\,0\leq \ell \leq 1, 
  \end{align}   where
\begin{align}
 G_0(x,y) :=-\frac{i}{2} \alpha \textrm{sgn}(x-y)- \frac{m}{2} (\beta+I)|x-y|=(D_m+mI)\bigg(\frac{-|x-y|}{2} \bigg),\label{eq:G0 defn}
 \end{align}
 \begin{equation}
 G_1(x,y) := \frac12 \alpha  (x-y)- \frac{im}{4} (\beta+I)|x-y|^2+\frac{i}{4m}I.\label{eq:G1 defn}
 \end{equation} 
More generally, for each $k=0,1,\ldots,$
 \be  \label{eq:R0exp_1}
  \mathcal{R}_0(z)(x,y)  = \frac{im}{2z}(\beta+I)+\sum_{j=0}^k z^j G_j(x,y)+O_2\big(z^{k+} \la r\ra^{k+1+}\big),  
  \ee  where  $G_j(x,y)=O(\la x-y\ra^{j+1})$.
 \end{lemma}
 \begin{proof} The first bound is immediate from the expansion in \eqref{resolventex}.
 For the others, it is easy to check the bounds  for $|z||x-y|<1$ by using the Taylor expansion of the exponential. 

Note that when $|z||x-y|>1$, we have  
 $$
 |\mR_0(z)-\frac{im}{2z}(\beta+I)-G_0  |\les \frac1{|z|}+ |x-y|\les \la r\ra\les z^\ell \la r\ra^{1+\ell},\,\,0\leq \ell \leq 1,
 $$
 $$  |\partial_z(\mR_0(z)-\frac{im}{2z}(\beta+I)-G_0)|\les \frac{|x-y|}z\les z^{\ell-1} \la r\ra^{1+\ell},\,\,0\leq \ell \leq 1,
$$
which implies the first bound for $|z||x-y|>1$. 
 
 Similarly, when $|z||x-y|>1$ and $j=0,1,2$
 $$
 |\partial_z^j(\mR_0(z)-\frac{im}{2z}(\beta+I)-G_0-zG_1)|\les  z^{1-j}|x-y|^2 \les z^{1+\ell-j} \la r\ra^{2+\ell},\,\,0\leq \ell \leq 1.  
 $$
The last bound follows similarly.
 \end{proof} 
 
Recalling the definition of $M(z)$ in \eqref{eqn:M defn}, we note that the contribution of the leading term $\frac{im}{2z}(\beta+I)$ of  $\mR_0 (z)$ to $M(z)$ is the operator with the kernel 
$$
\frac{im}{2z}v(x)(\beta+I)v^*(y)=\frac{im}{2z}\left[\begin{array}{c}a(x)+b(x)\\c(x)+d(x)\end{array}\right]  \big[\overline{a(y)+b(y)}, \overline{c(y)+d(y)}\big] =: g(z) P(x,y),
$$
where  
\begin{equation}
 g(z)= \frac{im }{2 z} \| (a+b,c+d)^T\|_{L^2}^{2}.\label{eq:g defn}
 \end{equation}
Here $P$ is the  orthogonal projection onto the span of the vector $(a+b, c+d)^T$, whose integral kernel is 
 $$P(x,y)=\|(a+b,c+d)\|_{L^2}^{-2}v(x)(\beta+I) v^*(y).$$
 
\begin{rmk} It is possible to have $(a+b,c+d)=0$ everywhere even if $V$ is not identically zero. 
	In this case $V$ has rank one, and since $V$ is self-adjoint it has to take the special form $V(x)=k(x) \begin{pmatrix}
	1 & -1\\ -1 & 1
	\end{pmatrix}$. Furthermore, $P=0$ and the threshold $m$ is regular. We ignore the details for this special case.  
\end{rmk}
 We also define the self-adjoint and absolutely bounded operator (provided $|v(x)|\les \la x\ra^{-\frac32-}$)
\begin{align}
T&:=U+vG_0v^*. \label{eq:T defn}   
\end{align}
The following expansions for $M(z)$ follows immediately from the expansions in Lemma~\ref{lem:R0exp} and the discussions  above noting that the operator with kernel 
$v(x)\la x-y\ra^s v^*(y)$ (where $s\geq 0$) is a Hilbert-Schmidt operator provided that $|v(x)|\les \la x\ra^{-s-\frac12-}$.  
\begin{lemma}\label{Mplus} Assume that $|v(x)|\les \la x\ra^{-\delta}$. Then 
$$
M(z)=g(z)P+\mathcal M_0(z), 
$$ 
where $\mathcal M_0(z)=\Gamma_0^0$ provided that $\delta>\frac12$. Moreover, 
$$
\mathcal M_0(z)=T+ \mathcal M_1(z), 
$$ 
where $ \mathcal M_1(z)=\Gamma^1_{0+},$ provided that $\delta>\frac32$. Furthermore, for each $k=1,2,\ldots,$
$$
\mathcal M_1(z)= \sum_{j=1}^k z^j M_j+\Gamma^2_{k+}, \text{ where } M_j=vG_jv^*,
$$
provided that $\delta>\frac32+k $.  
\end{lemma}

\begin{defin}\label{def:regular}
Let $Q=I-P$. 
We say that the threshold $\lambda=m$ is regular if $Q\mathcal M_0 Q$ is invertible on $QL^2$ for all $0<|z|<z_0$ for some $z_0>0$ and if $\|(Q\mathcal M_0 Q)^{-1}\|_{L^2\to L^2} $ remains bounded as $z\to 0$. Note that in the case  $|v(x)|\les \la x\ra^{-\frac32-}$, this is equivalent to the invertibility of  $T$ on $QL^2$ by the expansions in the Lemma~\ref{Mplus} above. In that case, we define the operator $D_0=(QTQ)^{-1}$.  Also note that, since $QTQ$ is a compact perturbation of $QUQ$, by Fredholm alternative $QTQ$ is invertible if and only if its kernel is empty.  If $T$ is not invertible,  with $S_1$ the Riesz projection onto its kernel, we instead define $D_0$ as  $D_0=[Q(T +S_1)Q]^{-1}$ with a slight abuse of notation. We note that in both cases $D_0$ is an absolutely bounded operator; the proof follows as in  Lemma~7.1 in \cite{egd} with minor modifications.  See Section~\ref{sec:spec} below for a classification of $S_1L^2$; in particular we prove that it has dimension at most one.
 \end{defin}

The following proposition establishes the invertibility of $M(z)$ for sufficiently small $z$ and provides expansions for its inverse in the case $m$ is a regular point of the spectrum. The first expansion will be useful in establishing the limiting absorption principle for energies close to the threshold $m$. The second expansion will be used in the proof of dispersive estimates for low energies with decay rate $|t|^{-\frac12}$, and the third one for the weighted estimates with improved  time decay. 

\begin{prop}\label{Minversenew} Assume that $m$ is a regular point of the spectrum,  and $|v(x)|\les \la x\ra^{-\tfrac12-}$.  Then, there exists $z_0>0$ so that   for all $0<|z|<z_0$, $M(z)$ is invertible on $L^2 $ and
	$$
	M^{-1}(z)=c_P z P + z^2\Lambda_0(z)+z \Lambda_1(z)Q+z Q\Lambda_2(z)+Q\Lambda_3(z)Q,
	$$
where $c_P= \frac{-2i    }{m \| (a+b,c+d)^T\|_2^{ 2} }$ and for $0<|z|<z_0$,  $\| \Lambda_j\|_{L^2\to L^2} \les 1$, for $j=0,1,2,3.$

Moreover, if $v$ decays faster, $|v(x)|\les \la x\ra^{-\tfrac32-}$,
then 
	$$
	M^{-1}(z)=z^2\Lambda_0(z)+ z \Lambda_1(z)Q+z Q\Lambda_2(z)+Q\Lambda_3(z)Q,
	$$
	where $\Lambda_j$   are absolutely bounded operators satisfying  
\be\label{Lambdat12}
	\|\, |z\Lambda_j(z)|\,\|_{L^2\to L^2}\in L^1_z, \,\,   \|\,|\partial_z(z\Lambda_j(z))|\,\|_{L^2\to L^2}\in L^1_z.
\ee
Furthermore, if $|v(x)|\les \la x\ra^{-\tfrac52-}$, then 
$$
	M^{-1}(z)=c_P z P + z^2\Lambda_0(z)+z \Lambda_1(z)Q+z Q\Lambda_2(z)+Q\Lambda_3(z)Q,
	$$
where    $\Lambda_j$   are absolutely bounded operators satisfying  the improved bounds
\be\label{Lambdat32}
	\|\,|\partial_z^k\Lambda_j(z)|\,\|_{L^2\to L^2}\in L^1_z, \quad k=0,1,2, \quad j=0,1,2,3.
\ee  
\end{prop}
\begin{proof} 
 By definition $Q\cM_0(z)Q$ is invertible on $QL^2$ for all $0<|z|<z_0$ with a uniformly bounded inverse. Let $D(z):= (Q\cM_0(z)Q)^{-1}$.
To invert $M(z)$, we write it with respect to the decomposition $L^2=PL^2\oplus QL^2$,
$$
M(z)=g(z)P+\cM_0(z) = \left[\begin{array}{cc} g(z)P+P\cM_0(z) P&   P \cM_0(z) Q\\  Q\cM_0(z) P &   Q \cM_0(z)  Q \end{array} \right]. 
$$  
Recall that  by the Fehsbach formula invertibility of a block matrix $A $ hinges upon the invertibility of $a_{22}$ and $(a_{11}-a_{12}a_{22}^{-1}a_{21})$. In this case, with $d=(a_{11}-a_{12}a_{22}^{-1}a_{21})^{-1}$, we have
    \begin{align}\nonumber
		A^{-1}&=\left[\begin{array}{cc} d & -da_{12}a_{22}^{-1}\\
		-a_{22}^{-1}a_{21}d & a_{22}^{-1}a_{21}da_{12}a_{22}^{-1}+a_{22}^{-1}
		\end{array}\right].
	\end{align} 
Note that in our case 
$$a_{22}^{-1}=(Q\cM_0(z)Q)^{-1}=D(z),\,\,\text{ and} $$
 $$ a_{11}-a_{12}a_{22}^{-1}a_{21}=P[g(z)+\cM_0(z)-\cM_0(z) Q   D(z)   Q  \cM_0(z) ]P =h(z)P,  \text{ where}$$
\be \label{hnew}
h(z)=g(z)+\text{Tr}(P\cM_0(z) P-P\cM_0(z) Q   D(z)   Q  \cM_0(z) P)=\frac{im\| (a+b,c+d)^T\|_2^{ 2}}{2z}+O(1),
\ee
provided that $|v(x)|\les \la x\ra^{-\tfrac12-}$ (by Lemma~\ref{Mplus}). 

Therefore, we see that $d(z)$ exists for sufficiently small but nonzero $z$:
\be\label{hnew2}
d(z)=\frac1{h(z)}P= \bigg[\frac{-2i    }{m \| (a+b,c+d)^T\|_2^{ 2} } z+ O(z^{2})\bigg]P=c_PzP+O(z^2). 
\ee

Using this in Feshbach formula, we have
\begin{multline}\label{Minv-gecici}
M^{-1}(z)=     Q   D(z)    Q +\\  
 \frac1{h(z)} \big(P-  P \cM_0(z) Q   D(z)   Q  -   Q   D(z)  Q\cM_0(z)  P +   Q   D(z)   Q\cM_0(z) P \cM_0(z)   Q  D(z)  Q\big)\\
=  \frac1{h(z)} P +z  \Lambda_1(z)Q+z Q\Lambda_2(z)+Q\Lambda_3(z)Q\\
= c_PzP+z^2 \Lambda_0(z) +z  \Lambda_1(z)Q+z Q\Lambda_2(z)+Q\Lambda_3(z)Q.
\end{multline}
Here each $\Lambda_j=\Gamma_0^0$. We have extra powers of $z$ next to $\Lambda_j$, $j=0,1,2$, coming from $\frac1{h(z)}$ and the formula \eqref{hnew2}.  
This proves the first expansion in the lemma. 

To obtain the other two expansions note that for $|v(x)|\les \la x\ra^{-\frac32-}$ we have $ \cM_0(z)=T+\cM_1(z)$, where $\cM_1(z)=\Gamma_{0+}^1$. Therefore,  $QTQ$ is invertible on $QL^2$ and $D_0=(QTQ)^{-1}$ is an  absolutely bounded operator on $QL^2$ (see the discussion in Definition~\ref{def:regular}).
Since $\cM_1(z)=O(z^{0+})$ as an Hilbert-Schmidt operator (see Lemma~\ref{Mplus}), we see   that    $D(z) =(Q\cM_0(z)Q)^{-1}$ is an absolutely bounded operator on $QL^2$ satisfying 
$\|D(z)\|\les 1.$ 
 Noting that
$$
\partial_zD(z)=D(z)(Q\partial_z \cM_0(z)Q) D(z),
$$  
$\partial_z T=0$ and using Lemma~\ref{Mplus}, we conclude
 $$
 \partial_z \cM_0(z)= \partial_z \cM_1(z) =\Gamma_{-1+}^0.
 $$ 
Hence
 $$
 \|\partial_zD(z)\|\les z^{-1+}.
 $$
Similarly, using the   bound  
$$
\mathcal M_1(z)= z M_1+\Gamma^2_{1+}, 
$$ 
from Lemma~\ref{Mplus}, 
which requires $|v(x)|\les \la x\ra^{-\tfrac52-}$, we obtain by a Neumann series expansion
\be\label{eq:Dexp} D(z)=D_0+z\Gamma+\Gamma^2_{1+}.
\ee
In particular, 
$$
\|\partial_{zz}D(z)\|\les z^{-1+}.
$$
Using these bounds in the definition of $h(z)$ we have
\be \label{hnew3}
h(z)=g(z)+\text{Tr}(P\cM_0(z) P-P\cM_0(z) Q   D(z)   Q  \cM_0(z) P)=\frac{im\| (a+b,c+d)^T\|_2^{ 2}}{2z}+c_0+O_1(z^{0+}),
\ee
provided that $|v(x)|\les \la x\ra^{-\tfrac32-}$ (by Lemma~\ref{Mplus} and the bounds for $QD(z)Q$ obtained above). 
Therefore, we see that 
\be\label{hnew2}
d(z)=\frac1{h(z)}P= \big[c_P z +c_2z^2+O_1(z^{2+})\bigg]P. 
\ee
In fact the error term  can be improved to $c_3z^3+O_2(z^{3+})$   if $|v(x)|\les \la x\ra^{-\tfrac52-}$ using \eqref{eq:Dexp} and the expansion for $\cM_0=T+\cM_1$ from Lemma~\ref{Mplus} with $k=1$ in \eqref{hnew}.  The exact values of unspecified constants are unimportant for our analysis.

Using these in Feshbach formula, we have
\begin{multline*}
M^{-1}(z)=     Q   D(z)    Q + \\
 \frac1{h(z)} \big(P-  P \cM_0(z) Q   D(z)   Q  -   Q   D(z)  Q\cM_0(z)  P +   Q   D(z)   Q\cM_0(z) P \cM_0(z)   Q  D(z)  Q\big)\\
=  \frac1{h(z)} P +z  \Lambda_1(z)Q+z Q\Lambda_2(z)+Q\Lambda_3(z)Q.
\end{multline*}
Here each $\Lambda_j$ satisfies the same bounds as $QD(z)Q$ obtained above, in particular \eqref{Lambdat12} and \eqref{Lambdat32},  and as before $\Lambda_1$, $\Lambda_2$ have one extra power of $z$ next to them coming from $\frac1{h(z)}$. Finally, using the expansion for $\frac1h$ used in \eqref{hnew2}
$$
\Lambda_0(z)=\frac1{z^2}( c_Pz + O_2(z^{1+}))P,$$ 
and hence $z\Lambda_0(z)= (c_P+O_2(z^{0+}))P$ and it  satisfies \eqref{Lambdat12}. 
In the case $v$ has more decay, we can write 
$$
\frac1{h(z)} P =c_P z P+ (c_2z^2+c_3z^3+O_2(z^{3+}))P=c_PzP +  z^2\Lambda_0(z),
$$
where $\Lambda_0(z)=(c_0 +c_1z +O_2(z^{1+}))P$ satisfies \eqref{Lambdat32}.
 \end{proof}

These expansions suffice to prove low energy dispersive bounds when the threshold is regular.  When the threshold is not regular, we develop  the following expansions.

\begin{prop}\label{MinversenewS_1} Assume that $m$ is not a regular point of the spectrum,  and $|v(x)|\les \la x\ra^{-\tfrac52-}$. Then, there exists $z_0>0$ so that   for all $0<|z|<z_0$, $M(z)$ is invertible on $L^2 $ and
$$
M^{-1}(z)=     z^2\Lambda_0(z)+ z \Lambda_1(z)Q+zQ\Lambda_2(z) +Q\Lambda_3(z)Q,
$$
where   each $\Lambda_j$   is an absolutely bounded operator satisfying the bounds \eqref{Lambdat12}. 
 
Furthermore, if $|v(x)|\les \la x\ra^{-\tfrac92-}$, we have
\be \label{Minvtemp2}
M^{-1}(z)= \frac{c_{-1}}{z} S_1+ S_1\Gamma P  +P\Gamma S_1+  c_1z  P   
 + z^2\Lambda_0(z)+  zQ\Lambda_1(z)+z \Lambda_2(z)Q+Q\Lambda_3(z)Q,
\ee
where  each $\Lambda_j$ satisfies \eqref{Lambdat32}.  
\end{prop}
\begin{proof} 
We start with the first assertion. 
	Let $S=P+S_1$ and $Q_1=Q-S_1$.  Note that, since $T$ is self adjoint and $S_1$ is the Riesz projection onto its kernel,  $Q_1TQ_1$ is invertible. In addition, $D_1=(Q_1TQ_1)^{-1}$ is absolutely bounded on $Q_1L^2$.  This is seen by noting that the resolvent identity and  $Q_1S_1=0$ yield 
	$Q_1D_0Q_1=Q_1D_1Q_1$. Using $Q=Q_1+S_1$ and $S_1D_0=D_0S_1=S_1$, we see that $Q_1D_0Q_1=QD_0Q-S_1$.  Combining these we see that $D_1=Q_1D_1Q_1=QD_0Q-S_1$, so that $D_1$ is the difference of two absolutely bounded operators.
Note that $Q_1M(z)Q_1=Q_1\cM_0(z)Q_1=Q_1TQ_1+\Gamma^1_{0+}$ is invertible on $Q_1L^2$ and $D(z)= (Q_1\cM_0(z)Q_1)^{-1}$ satisfies the same bounds as $D(z)$ in the proof of Proposition~\ref{Minversenew}. In fact, using the  decay of $v$ we have 
$$D(z)=D_0+z\Gamma+ \Gamma^2_{1+}.
$$
We write 
$$
M(z)=g(z)P+\cM_0(z) = \left[\begin{array}{cc} SM(z)S&   S \cM_0(z) Q_1\\  Q_1\cM_0(z) S &   Q_1 \cM_0(z)  Q_1 \end{array} \right]. 
$$ 
We claim  that 
 $B(z):= a_{11}-a_{12}a_{22}^{-1}a_{21}=S[M(z) -\cM_0(z) Q_1   D(z)   Q_1  \cM_0(z) ]S $ is invertible on $SL^2$ for small but nonzero $z$ and denote its inverse by $d(z)$. 
Therefore by the  Feshbach formula, we have
\begin{multline}\label{Minvtemp1}
M^{-1}(z)=    Sd(z)S+  Q_1   D(z)    Q_1 +Q_1D(z)Q_1\cM_0(z)Sd(z)S\cM_0(z)Q_1D(z)Q_1 \\
- S d(z)S\cM_0(z)Q_1D(z)Q_1- Q_1D(z)Q_1\cM_0(z)Sd(z)S.
\end{multline}
We now  prove that $B$ is invertible as claimed. Recalling that 
$$S_1,Q_1\leq Q,\,\,\,S_1TQ=S_1Q_1=S_1P=PQ=0, \,\,\, M(z)=g(z)P+\cM_0(z),\, \,\,  \cM_0(z) =T+\cM_1(z),
$$  and dropping $z$ dependence, we write
\begin{multline}\label{eqn:B matrix def}
B  = S[M(z) -\cM_0(z) Q_1   D(z)   Q_1  \cM_0(z) ]S\\= \left[\begin{array}{cc} S_1[\cM_1 -\cM_1 Q_1   D   Q_1  \cM_1 ]S_1&   S_1[T+\cM_1  -\cM_1 Q_1   D    Q_1  \cM_0  ]P\\ P[T+\cM_1  -\cM_0 Q_1   D   Q_1  \cM_1  ]S_1 & h  P  \end{array} \right]\\=
\left[\begin{array}{cc} zS_1M_1S_1   +O_2(z^{1+})S_1 &   S_1[T+z\Gamma^*  +O_2(z^{1+}) ]P\\ P[T+z\Gamma + O_2(z^{1+}) ]S_1 & h  P  \end{array} \right]
, 
\end{multline} 
where $h$ is as in \eqref{hnew} with $Q$ replaced with $Q_1$; it satisfies the same expansions as before. 
Since $PL^2$ and $S_1L^2$ are one dimensional subspaces, see Corollary~\ref{cor:S1} below,  we can choose unit $\phi\in  S_1L^2$ and $\theta=\|(a+b,c+d)^T\|_2^{-1} v(1,1)^T\in PL^2$, and writing  $B$ with respect to the basis $\{\phi,\theta\}$ we have  
\be\label{eq:BB-1}
B=\left[\begin{array}{cc} k &   \overline{\ell}\\ \ell & h   \end{array} \right],\,\,\,B^{-1}=\frac1{hk-|\ell|^2}\left[\begin{array}{cc} h &   -\overline{\ell}\\ -{\ell} & k   \end{array} \right],
\ee
where 
\be\label{eq:k} k 
=z\text{Tr}(S_1M_1S_1)+ O_2(z^{1+}),\ee
\be\label{eq:ell}
\ell=
\la T\phi, \theta\ra+c_1z +O_2(z^{1+})=\kappa_0\|(a+b,c+d)^T\|_2+c_1z +O_2(z^{1+}),
\ee
and $\kappa_0$ is as in Lemma~\ref{lem:S1 char}. Hence 
\be\label{eq:ell2}
|\ell|^2=\text{Tr}(S_1TPTS_1)+c_1z+ O_2(z^{1+})=|\kappa_0|^2\|(a+b,c+d)^T\|_2^2+c_1z+ O_2(z^{1+}).
\ee
Therefore,
\be\label{eq:det}
hk-|\ell|^2= \frac{im\| (a+b,c+d)^T\|_2^{ 2}}{2 }\big( \text{Tr}(S_1M_1S_1) +\frac{2i}{m}|\kappa_0|^2\big)
+ O_2(z^{0+}).
\ee
The leading term is nonzero by Lemma~\ref{prop:B1 inv} below, and hence $hk-|\ell|^2\neq 0$ for small $z$.
Therefore, we obtain
$$
d= B^{-1}=(\frac{c_{-1}}{z}+ O_2(z^{-1+})) S_1 +( c_1z + O_2(z^{1+}))P\\ + (c_0+ O_2(z^{0+}) ) S_1\Gamma P+(c_0+ O_2(z^{0+})) P \Gamma  S_1.
$$
Using this in \eqref{Minvtemp1} and noting that 
\be\label{temporth}
S_1\leq Q,\,\,\,\,  Q_1\cM_0S_1= Q_1\cM_1S_1 =Q_1 [z\Gamma +O_2(z^{1+})]S_1,
\ee
we obtain
$$
M^{-1}(z)= z^2\Lambda_0(z)
 +   zQ\Lambda_1(z)+z \Lambda_2(z)Q+Q\Lambda_3(z)Q,
$$ 
where $\Lambda_j$'s satisfy the bounds in \eqref{Lambdat12}. Indeed, $\Lambda_0(z)=\frac1{z^2}( c_1z + O_2(z^{1+}))P$, and hence $z\Lambda_0(z)= (c_1+O_2(z^{0+}))P$ and it  satisfies \eqref{Lambdat12}. Similarly, the most singular term of 
$\Lambda_3$ is $(\frac{c_{-1}}{z}+ O_2(z^{-1+})) S_1$, and hence $z\Lambda_3(z)=( c_{-1} + O_2(z^{0+})) S_1+\cdots $ satisfies \eqref{Lambdat12}. The other terms are controlled similarly.

For the second assertion, we follow the same proof expanding  the operators to higher orders of $z$. Using the additional decay of $v$ we have (recall the constants and any operator denoted $\Gamma$ are allowed to change from line to line and even in the same line)
$$D(z)=D_0+z\Gamma+z^2\Gamma+z^3\Gamma+\Gamma^2_{3+}, \text{ and} 
$$
$$\cM_0=T+zM_1+z^2M_2+z^3M_3+O_2(z^{3+}).$$ Using these we have the following expansions for $ k $ and $\ell$ (ignoring the actual form of the leading terms): 
$$k
=c_1z+c_2z^2+c_3z^3+O_2(z^{3+}),$$
$$\ell=
c_0+c_1z+c_2z^2+O_2(z^{2+}),\,\,\, 
|\ell|^2=c_0+c_1z+c_2z^2+O_2(z^{2+}), \text{ and hence} 
$$ 
$$
hk-|\ell|^2= c_0+c_1z+c_2z^2+O_2(z^{2+}).
$$
Therefore, we obtain
\begin{multline*}
d= B^{-1}=(\frac{c_{-1}}{z}+c_0+c_1z+O_2(z^{1+})) S_1 +( c_1z +c_2z^2+c_3z^3+O_2(z^{3+}))P\\ + (c_0+c_1z+c_2z^2+O_2(z^{2+}) ) S_1\Gamma P+(c_0+c_1z+c_2z^2+O_2(z^{2+})) P \Gamma  S_1.
\end{multline*}
Using this and \eqref{temporth} in \eqref{Minvtemp1}, we obtain
$$
M^{-1}(z)= \frac{c_{-1}}{z} S_1 +  S_1\Gamma P  +P\Gamma S_1+   c_1z  P 
 + z^2\Lambda_0(z)+  zQ\Lambda_1(z)+z \Lambda_2(z)Q+Q\Lambda_3(z)Q,
$$
where $\Lambda_j$'s satisfy the bounds $\|\,|\partial_z^k\Lambda_j(z)|\,\|_{L^2\to L^2}\in L^1_z, k=0,1,2.$
\end{proof}

\begin{rmk}\label{rmk:Binv leading}
	
	In fact, to prove the final result in Theorem~\ref{thm:main nonreg}, an exact accounting of the leading terms in $B^{-1}$ is required.  Using the leading terms of  \eqref{hnew}, \eqref{eq:k}, \eqref{eq:ell}, \eqref{eq:ell2}, and \eqref{eq:det} in \eqref{eq:BB-1},
we can write  the leading terms of $B^{-1}$ as
$$
\frac1{\mathcal D}\Big[
\frac{1}{c_Pz }  S_1 -S_1TP-PTS_1+z  \text{Tr}(S_1M_1S_1) P\Big],
$$
where 
$$\mathcal D=\frac1{c_P}\big( \text{Tr}(S_1M_1S_1) +\frac{2i}{m}|\kappa_0|^2\big), \,\,\text{ and}\,\,\,
c_P= \frac{-2i    }{m \| (a+b,c+d)^T\|_2^{ 2} }.$$  
We can rewrite the expansion above as	
$$ c_PzP+ \frac1{\mathcal D}\Big[\frac{1}{c_Pz }S_1-S_1TP-PTS_1-\frac{2i}{m}|\kappa_0|^2zP\Big].
$$
Therefore, we can rewrite \eqref{Minvtemp2} as
\be\label{Minvfinal?}
M^{-1}(z)= c_PzP+ \frac1{\mathcal D}\Big[\frac{1}{c_Pz }S_1-S_1TP-PTS_1-\frac{2i}{m}|\kappa_0|^2zP\Big] + z^2\Lambda_0(z)+  zQ\Lambda_1(z)+z \Lambda_2(z)Q+Q\Lambda_3(z)Q.
\ee
\end{rmk}

\section{Low energy dispersive bounds}\label{sec:low disp}

In this section we prove the low energy bounds.  We begin with the unweighted bound when the threshold energy is regular.  In all cases, we utilize the Stone's formula and extend to the real line as usual to bound
\begin{align}
	e^{-itH}P_{ac}(H)
	=\frac{1}{2\pi i} \int_{\R} e^{-it\sqrt{z^2+m^2}} \frac{z}{\sqrt{z^2+m^2}}  \mR_V^+ (z)(x,y)\, dz.
\end{align}
Then, we appeal to the symmetric resolvent identity \eqref{eqn:symm res id}
and employ Lemma~\ref{lem:vdc}.  After extending to the real line, we omit the `+' on the resolvent operators.
We let $\chi(z)$ be a smooth, even cut-off satisfying $\chi(z)=1$ when $|z|\leq \frac{z_0}{2}$ and $\chi(z)=0$ when $|z|\geq |z_0|$.  Here $z_0$ is the minimum of the constants from Proposition~\ref{Minversenew} and \ref{MinversenewS_1}.

\begin{prop}\label{prop:low disp reg}
	
	Assuming that $|v(x)|\les \la x \ra^{-\delta}$ for some $\delta>\f32$, if the threshold energies are regular, we have the following bound
	$$
		\|e^{-itH}P_{ac}(H)\chi(H)\|_{L^1\to L^\infty} \les \la t\ra^{-\f12}.
	$$
	Further, if $\delta>\f52$ we have (for $|t|>1$)
	$$
		\big|[e^{-itH}P_{ac}(H)\chi(H)](x,y)\big|  \les | t|^{-\f12} \min(1, | t|^{-1} \la x \ra \la y \ra ).
	$$
	
\end{prop}

We note that the leading term in the symmetric resolvent identity \eqref{eqn:symm res id} involving only the free resolvent may be controlled by Theorem~\ref{thm:giggles}.  To understand the contribution of the second term in \eqref{eqn:symm res id}, we note the expansions of $M^{-1}(z)$ in Proposition~\ref{Minversenew}.  Proposition~\ref{prop:low disp reg} suffices to establish the bounds of Theorem~\ref{thm:main nonreg} when $\tau=0$ and $\tau=1$.  The full range follows from interpolation.

The following proposition suffices to prove the first bound in Proposition~\ref{prop:low disp reg}, and takes care of the contribution of all terms in the second and third expansions in Proposition~\ref{Minversenew} needed for the second bound, except $c_PzP$: 
\begin{prop}\label{prop:Q} Assume that $\Lambda(z)$ is an absolutely bounded operator satisfying for $0<|z|<z_0$ the bounds \eqref{Lambdat12}, i.e.,
$$
\|\, |z\Lambda(z)|\,\|_{L^2\to L^2}\in L^1_z, \,\,   \|\,|\partial_z(z\Lambda(z))|\,\|_{L^2\to L^2}\in L^1_z.
$$ 
Then
$$
\bigg| \int_{\R} e^{-it\sqrt{z^2+m^2}} \frac{z\chi(z)}{\sqrt{z^2+m^2}} [\mR_0 (z)  vQ\Lambda(z)Qv^* \mR_0 (z)](x,y)\, dz \bigg| \les \la t\ra^{-\f12}.
$$
provided that $|v(x)|\les \la x\ra^{-\frac32-}.$
Moreover, if 	$\Lambda(z)$ satisfies the bounds \eqref{Lambdat32}, i.e.,   $\|\,|\partial_z^k\Lambda(z)|\,\|_{L^2\to L^2}\in L^1(z)$ for $k=0,1,2$,     and $|v(x)|\les \la x\ra^{-\frac52-}$, then, for $|t|>1$, the integral above is bounded by $|t|^{-\f12} \min (1, | t|^{-1}\la x\ra \la y \ra ) .$

Finally, the claims above remain valid under the same conditions on $\Lambda$ and $v$ if we replace $Q\Lambda(z)Q$ with $z\Lambda(z)Q$, or $zQ\Lambda(z)$, or $z^2\Lambda(z)$. 
\end{prop}
\begin{proof}
We begin by using \eqref{resolventex} to write
\begin{align}\label{eq:R0 R1}
	\mR_0(z)(x,y)=\frac{im}{2z}(\beta+I)e^{iz|x-y|}+e^{iz|x-y|} \mR_{1}(z).
\end{align}
Here $\mR_1(z)$ is the non-singular portion of the free resolvent which satisfies $|\partial_z^k \mR_1(z)(x,y)|\les 1$ for $k=0,1,2$.

We first consider the most singular term involving $(\beta+I)/z$ on both sides.  Ignoring the constants we need to control:
$$	
	\int_{\R^3} e^{-it\sqrt{z^2+m^2}+iz(|x-x_1|+|y-y_1|)} \frac{\chi(z)}{z\sqrt{z^2+m^2}} \big[(\beta+I)   v^*Q\Lambda(z)Qv (\beta+I)\big](x_1,y_1)\, dx_1\,  dy_1 \,  dz.
$$
Noting that $Qv(\beta+I)=(\beta+I)v^*Q=0$, i.e.
	\be \label{eqn:Q}
	\int_{\R}  (\beta+I)v^*(x_1) Q(x_1,x_2)\, dx_1= \int_{\R}  Q(y_2,y_1)v(y_1)(\beta+I)\, dy_1=  0,
	\ee
we can rewrite the   integral as
	$$ 
	\int_{\R^3 }e^{-it\sqrt{z^2+m^2} } \frac{\chi(z)}{z \sqrt{z^2+m^2}} \bigg[e^{iz|x-x_1|} -e^{iz|x|} \bigg]  [(\beta+I) 
	v^* Q\Lambda(z)Q v  (\beta+I)](x_1,y_1) \bigg[ e^{iz|y_1-y|}-e^{iz|y|} \bigg] \, dx_1\, dy_1\, dz.
$$
	Writing
	\begin{align}\label{eqn:F defn}
	e^{iz|x-x_1|} -e^{iz|x|}=iz \int_{|x|}^{|x-x_1|} e^{izs_1}\, ds_1:=iz F(z,x,x_1),
	\end{align}
	and similarly for the second difference of phases, then changing the order of integration, we need to control (with $\widetilde \Lambda(z,x_1,y_1)= [(\beta+I) 
	v^* Q\Lambda(z)Q v  (\beta+I)](x_1,y_1)$)
\begin{align}\label{eqn:Lambda twiddle}
	\int_{\R^2 }\int_{|x|}^{|x-x_1|} \int_{|y|}^{|y-y_1|}	\int_{\R  } e^{-it\sqrt{z^2+m^2}+iz(s_1+s_2) } \frac{z\chi(z)}{  \sqrt{z^2+m^2}}  \widetilde \Lambda(z,x_1,y_1) \, dz\,ds_1\,ds_2  \, dx_1\,dy_1.
\end{align}
We use Lemma~\ref{lem:vdc} (for $j=0$) in the $z$ integral with $\psi_0(z)=\frac{z\chi(z)}{  \sqrt{z^2+m^2}}  \widetilde \Lambda(z,x_1,y_1)$, to obtain the bound
\be\label{minfromvdc}
|\eqref{eqn:Lambda twiddle}|
\les\min\Big(\| z \widetilde \Lambda\|_{L^1_{|z|\ll1}},|t|^{-\f12}\big\|\partial_z(z   \widetilde \Lambda(z))\big\|_{L^1_{|z|\ll 1}}, |t|^{-\f32} \big\| 
|\widetilde \Lambda_{zz}|+(s_1+s_2) |\widetilde \Lambda_z|+| \widetilde \Lambda| \big\|_{L^1_{|z|\ll1}}\Big)
\ee
We consider the contribution of the last bound only since the other two follow similarly:
\be\label{tempt32}
|t|^{-\frac32}
\int_{\R^2 }\int_{|x|}^{|x-x_1|} \int_{|y|}^{|y-y_1|}	\int_{|z|\ll 1}  \big(|\widetilde \Lambda_{zz}|+(s_1+s_2) |\widetilde \Lambda_z|+| \widetilde \Lambda|\big)\, dz \,ds_1\,ds_2\,  dx_1 \, dy_1.
\ee
Note that 
\begin{align}\label{eqn:F wt small}
		\int_{|x|}^{|x-x_1|} \int_{|y|}^{|y-y_1|} 1 \,  ds_1\, ds_2\les \la x_1 \ra \la y_1 \ra,
\end{align}
as the length of the $s_1$ and $s_2$ integrals may be bounded by $|\,|x-x_1|-|x|\,|\les \la x_1\ra $ and $\la y_1\ra $ respectively. 
We also have
\begin{align}\label{eqn:F wt big}
\int_{|x|}^{|x-x_1|} \int_{|y|}^{|y-y_1|} 
	(s_1+s_2)  \,  ds_1\, ds_2\les   \la x_1 \ra^2 \la y_1 \ra^2 \la x\ra \la y \ra.
\end{align}	
Here we used that $|\int_a^b s\, ds|=\frac{1}{2}|b^2-a^2|\les |(b-a)(b+a)|\les |b-a|\max (a,b)$.
Therefore, after changing the order of integration, we see
$$
	|\eqref{tempt32}|\les
	|t|^{-\frac32}
	\int_{|z|\ll 1}	\int_{\R^2 } \la x_1\ra\la y_1\ra \big(|\widetilde \Lambda_{zz}|+   \la x_1 \ra \la y_1 \ra \la x\ra \la y \ra
	|\widetilde \Lambda_z|+| \widetilde \Lambda|\big)\, dx_1dy_1dz.
$$
Recalling the definition of $\widetilde \Lambda$ and using the Cauchy-Schwarz inequality in $x_1$ and $y_1$ integrals, we obtain
\begin{multline*}
	|\eqref{eqn:Lambda twiddle}| 	\les |t|^{-\frac32}
	\int_{|z|\ll 1} \|\la x_1\ra |v^*(x_1)|\|_{L^2_{x_1}} (\| \,|\Lambda_{zz}|\,\|_{L^2\to L^2} + \| \,|\Lambda| \,\|_{L^2\to L^2} )\| \la y_1 \ra |v(y_1)|\|_{L^{2}_{y_1}} dz \\+ \la x\ra \la y \ra |t|^{-\frac32}
	\int_{|z|\ll 1} \|\la x_1\ra^2 |v^*(x_1)|\|_{L^2_{x_1}} \| \,|\Lambda_{z}|\,\|_{L^2\to L^2} \| \la y_1 \ra^2 |v(y_1)|\|_{L^{2}_{y_1}} dz 
	\les  	\la x\ra \la y\ra \la t\ra^{-\f32}.
\end{multline*}
In the last inequality, we used $|v(x)|\les \la x\ra^{-\f52-}$ and the bounds on $\Lambda$ and its derivatives.  Note that the contribution of the first two terms in \eqref{minfromvdc} can be handled similarly but only requires $|v(x)|\les \la x\ra^{-\frac32-}$ as we need only use \eqref{eqn:F wt small} and not \eqref{eqn:F wt big}.

We now consider the least singular term involving $\mR_1(z)$ on both sides. 
We need to control:
$$	\int_{\R^3} e^{-it\sqrt{z^2+m^2}+iz(|x-x_1|+|y-y_1|)} \frac{z\chi(z)}{\sqrt{z^2+m^2}} \mR_1(z)(x,x_1)   [v^*Q\Lambda(z)Qv](x_1,y_1) \mR_1(z)(y_1,y)\, dx_1 dy_1 dz.
$$
Let $\widetilde \Lambda:= \mR_1(z)(x,x_1)   [v^*Q\Lambda(z)Qv](x_1,y_1) \mR_1(z)(y_1,y) $. 
After changing the order of integration, we use Lemma~\ref{lem:vdc} (for $j=0$) in the $z$ integral with $\psi_0(z)=\frac{z\chi(z)}{  \sqrt{z^2+m^2}}  \widetilde \Lambda $ to obtain the bound
\be\label{minfromvdc1}
\les\min\Big(\| z \widetilde \Lambda\|_{L^1_{|z|\ll1}},|t|^{-\f12}\big\|\partial_z(z   \widetilde \Lambda(z))\big\|_{L^1_{|z|\ll 1}}, |t|^{-\f32} \big\| 
|\widetilde \Lambda_{zz}|+(|x-x_1|+|y-y_1|) |\widetilde \Lambda_z|+| \widetilde \Lambda| \big\|_{L^1_{|z|\ll1}}\Big).
\ee
Since $|\partial_z^k \mR_1(z)|=O(1)$, $k=0,1,2$, we can assume all derivatives hit $\Lambda(z)$, and the proof proceeds as in the previous case. 

The remaining cases  are handled similarly using the additional factor(s) of $z$ in place of the missing $Q$ orthogonalities.   
\end{proof}

It now suffices to consider the contribution of the operator $c_PzP$  in \eqref{eqn:symm res id} for the weighted bound:
\begin{lemma}\label{lem:S disp ests}	
If $|v(x)|\les \la x\ra^{-\frac52-}$, then for $|t|>1$ we have the expansion
	\begin{align*}
		\frac{1}{2\pi i}\int_{\R} e^{-it\sqrt{z^2+m^2}} \frac{\chi(z) z }{ \sqrt{z^2+m^2}}  [\mR_0 (z)  v^*(c_PzP)v \mR_0 (z)](x,y)\, dz= F_t^0(x,y)+O(|t |^{-\f32}\la x \ra \la y \ra),
	\end{align*}
	where $F_t^0(x,y)$ is given by \eqref{Ft0def}.
\end{lemma}

\begin{proof} 
Using \eqref{eq:R0 R1} and recalling that $c_P=\frac{-2i    }{m \| (a+b,c+d)^T\|_2^{ 2} }$, we have
	\begin{multline*}
	 \mR_0 v^*(c_Pz P)v \mR_0\\= \frac{ -2iz}{m \| (a+b,c+d)^T\|_2^{ 2} }    e^{iz(|x-x_1|+|y_1-y|)}  
\bigg[\frac{im}{2z}(\beta+I)+ \mR_{1}(z)\bigg] v^*Pv \bigg[\frac{im}{2z}(\beta+I)+ \mR_{1}(z)\bigg]\\
		=\frac{  im  e^{iz(|x-x_1|+|y_1-y|)} }{2 \| (a+b,c+d)^T\|_2^2 z} (\beta+I)v^*Pv(\beta+I)+ e^{iz(|x-x_1|+|y_1-y|)} \mathcal E(z,x,y),
	\end{multline*}
	where $\mathcal E(z,x,y)$ satisfies $\| \partial_z^k  \mathcal E(z,x,y) \|_1\les 1$ for $k=0,1,2$.

Hence, applying Lemma~\ref{lem:vdc} to the contribution of the second term with $\psi_0(z)=\frac{z\chi(z)}{\sqrt{z^2+m^2}} \mathcal E(z,x,y)$ yields the desired bound.  
	To control the contribution of the first term, first note that 
	\begin{multline}\label{eqn:S P int}
	\int (\beta+I)v^*(x_1)P(x_1,y_1)v(y_1) (\beta+I)\, dx_1\, dy_1=\int (\beta+I)v^*(x_1) v(x_1) (\beta+I)\, dx_1  \\
	=\| (a+b,c+d)\|_2^2 (\beta+I).
	\end{multline}
We can rewrite the contribution of the first term as 
	\begin{multline*}
		\frac{m}{4 \pi  \| (a+b,c+d)^T\|_2^{ 2} }
		\int_{\R} e^{-it\sqrt{z^2+m^2}} \frac{\chi(z)}{\sqrt{z^2+m^2}}    e^{iz(|x-x_1|+|y-y_1|)} (\beta+I)v^*Pv(\beta+I) \, dz\\
		=\frac{m}{4\pi\| (a+b,c+d)^T\|_2^{ 2} }\int_{\R} e^{-it\sqrt{z^2+m^2}} e^{iz|x-y|}\frac{\chi(z)}{\sqrt{z^2+m^2}}      (\beta+I)v^*Pv(\beta+I) \, dz\\
		+\frac{m}{4\pi\| (a+b,c+d)^T\|_2^{ 2} }\int_{\R} e^{-it\sqrt{z^2+m^2}} \frac{ \chi(z) }{\sqrt{z^2+m^2}}   \bigg[e^{iz(|x-x_1|+|y-y_1|)}-e^{iz|x-y|} \bigg] (\beta+I)v^*Pv(\beta+I) \, dz.
	\end{multline*}
	Using \eqref{eqn:S P int} on the first term yields the operator $F_t^0(x,y)$, see \eqref{Ft0def}.  As in the proof of Proposition~\ref{prop:Q}, noting that 
	$$ 
		e^{iz(|x-x_1|+|y-y_1|)}-e^{iz|x-y|} =iz \int_{|x-y|}^{|x-x_1|+|y-y_1|} e^{isz}\, ds.
	$$ 
  allows us to apply Lemma~\ref{lem:vdc} to the second term. We note that
$$
\big|(|x-x_1|+|y-y_1|)^2-|x-y|^2\big|\les  \la x\ra \la y\ra \la x_1\ra^2 \la y_1\ra^2, 
$$
which leads to the weighted bound as in Proposition~\ref{prop:Q}.	
\end{proof}

We are now ready to prove Proposition~\ref{prop:low disp reg}.

\begin{proof}[Proof of Proposition~\ref{prop:low disp reg}]
	
	Utilizing the symmetric resolvent identity, \eqref{eqn:symm res id}, we bound the contribution of each term to \eqref{eqn:stone}.  For the unweighted bound, we note that Theorem~\ref{thm:giggles} controls the contribution of the sole free resolvent's contribution to show
	$$
	\bigg| \frac{1}{2\pi i} \int_{\R} e^{-it\sqrt{z^2+m^2}} \frac{z}{\sqrt{z^2+m^2}}  \mR_0  (z)(x,y)\, dz\bigg| \les \la t\ra^{-\f12}.	
	$$
	Using the second expansion for $M^{-1}(z)$ obtained in Proposition~\ref{Minversenew}, all summands are controlled by Proposition~\ref{prop:Q} for the unweighted bound. 
	
	For the  weighted bound, using the third expansion for $M^{-1}(z)$ obtained in Proposition~\ref{Minversenew}, all summands are controlled by Proposition~\ref{prop:Q} except the leading term $c_PzP$. One needs to utilize the delicate cancellation between this term and the free resolvent.   Specifically, using the second bound in Theorem~\ref{thm:giggles} for the free resolvent, and  Lemma~\ref{lem:S disp ests} for the contribution of $c_PzP$, we see that their contributions add up to $ F^0_t(x,y)+O
	(|t|^{-\f32 }  \la x-y\ra\big)-\big[F_t^0(x,y)+O(|t|^{-\f32}\la x \ra \la y \ra )\big] =O(|t|^{-\f32}\la x \ra \la y \ra ).$
\end{proof}

We now turn to the dispersive bounds when the threshold is not regular. Before we state the main result we have the following expansion for $F^0_t$:
\begin{lemma}\label{lem:Ft} For $|t|>1$, the operator  $F^0_t$ given by \eqref{Ft0def} in Lemma~\ref{lem:S disp ests} and Theorem~\ref{thm:giggles} can be written as
$$
F^0_t(x,y) = 
\frac{m}{4\pi }\frac{(-2\pi i)^{\frac12}e^{imt}}{(mt)^{\frac12}} (\beta+I)\chi(x /t)e^{-im\sqrt{t^2-x^2}}\chi(y /t)e^{-im\sqrt{t^2-y^2}}  +O(|t|^{-\frac32}\la x\ra \la y\ra).
$$
In particular, for $|t|>1$, it is the sum of a rank one operator satisfying the unweighted bound $|t|^{-\frac12}$ and an operator satisfying the weighted bound. 
\end{lemma}
\begin{proof}
Note that 
\begin{multline*}
\int_{\R} e^{-it\sqrt{z^2+m^2}}\big[e^{+iz|x-y|}-e^{iz(|x|+|y|)}\big] \frac{\chi(z)}{\sqrt{z^2+m^2}}    \, dz 
	\\=i\int^{|x|+|y|}_{|x-y|} \int_\R e^{-it\sqrt{z^2+m^2}+izs} \frac{z \chi(z)}{\sqrt{z^2+m^2}}  \, dz\,  ds 
	=O(|t|^{-\frac32}\la x\ra \la y\ra),
	\end{multline*}	
by Lemma~\ref{lem:vdc} noting that 
	$$
	\int^{|x|+|y|}_{|x-y|} s\,  ds =|xy|+xy.
	$$
	Defining 
	\be\label{Gtdef}
	G_t(r):=   \int_{\R} e^{-it\sqrt{z^2+m^2}+izr} \frac{\chi(z)}{\sqrt{z^2+m^2}}    \, dz, 
	\ee
	it remains to prove that $G_t(|x|+|y|)$,
	is finite rank up to an operator satisfying the weighted decay (for $|t|>1$). Note that if $r\gtrsim |t|$ then there are no critical points and by non-stationary phase, $G_t=O(r^{-N})=O(|t|^{-N})$. When $|x|,|y|\ll |t|$, by Lemma 3.6 in \cite{EGT}, with $r_3=r_4=0$,  we see
	$$
	G_t(|x|+|y|)=\frac{(-2\pi i)^{\frac12}e^{imt}}{(mt)^{\frac12}}e^{-im\sqrt{t^2-x^2}} e^{-im\sqrt{t^2-y^2}} + O(|t|^{-\frac32} \la x\ra \la y\ra).
	$$
	Therefore,
	\be\label{Gtexp}
	G_t(|x|+|y|)=\frac{(-2\pi i)^{\frac12}e^{imt}}{(mt)^{\frac12}}\chi(x/t) e^{-im\sqrt{t^2-x^2}} \chi(y/t)e^{-im\sqrt{t^2-y^2}} + O(|t|^{-\frac32} \la x\ra \la y\ra).
	\ee
Recalling \eqref{Ft0def}, we obtain the claim. 
\end{proof}

We are now ready to prove the first two claims in Theorem~\ref{thm:main nonreg}.
\begin{prop}\label{prop:nonreg}
	If the threshold energy $ m $ is not regular, 
we have 
	$$
	\|e^{-itH}P_{ac}(H)\chi(H)\|_{L^1\to L^\infty} \les \la t\ra^{-\f12},
	$$
	provided that $|v(x)|\les \la x\ra^{-\frac52-}$.
	Furthermore, if $|v(x)|\les \la x\ra^{-\f92-}$, then for $|t|>1$, there is an operator $ {F_t^+}(x,y)$ of rank at most one satisfying $\|{F_t^+}\|_{1\to \infty}\les | t|^{-\f12}$, so that
	$$
	\|e^{-itH}P_{ac}(H)\chi(H)-{F_t^+} \|_{L^1\to L^\infty} \les |t|^{-\f32} \la x\ra \la y \ra.
	$$
\end{prop} 
\begin{proof}
The unweighted bound follows from Proposition~\ref{prop:Q} and the first expansion in Proposition~\ref{MinversenewS_1}.  For the weighted bound, the contribution of the $\Lambda_j$ terms in the second expansion in Proposition~\ref{MinversenewS_1} are controlled by Proposition~\ref{prop:Q}.  Noting the expansion in \eqref{Minvfinal?}, we  need to understand the contribution of 
 $$c_PzP+ \frac1{\mathcal D}\Big[\frac{1}{c_Pz }S_1-S_1TP-PTS_1-\frac{2i}{m}|\kappa_0|^2zP\Big].  $$
We note that the contribution of the first term, $c_PzP$,  given by  Lemma~\ref{lem:S disp ests} exactly cancels with the free evolution as in the regular case. The contribution of $P$ in the second term gives $-\frac{2i}{{\mathcal D} mc_P}|\kappa_0|^2F_t^0$ by Lemma~\ref{lem:S disp ests}, which can be  controlled by 
Lemma~\ref{lem:Ft}  up to  a rank one operator:
\be\label{eqn:zP weird}
	-\frac{2i}{{\mathcal D} mc_P}|\kappa_0|^2F_t^0= \frac{|\kappa_0|^2}{2\pi i {\mathcal D} c_P}  \frac{(-2\pi i)^{\frac12}e^{imt}}{(mt)^{\frac12}}
	(\beta+I)\chi(x /t)e^{-im\sqrt{t^2-x^2}}\chi(y /t)e^{-im\sqrt{t^2-y^2}}
	+O(|t|^{-\frac32}\la x\ra \la y\ra).
\ee
We now consider the contribution of the remaining terms $\frac1{\mathcal D}\big[\frac{1}{c_Pz }S_1-S_1TP-PTS_1\big]$:
$$
\frac1{2\pi i \mathcal D}	\int_{\R} e^{-it\sqrt{z^2+m^2}} \frac{z\chi(z)}{\sqrt{z^2+m^2}} \Big[\mR_0(z)  v^* \big[\frac{1}{c_Pz }S_1-S_1TP-PTS_1\big]v \mR_0(z)\Big](x,y)\, dz.
$$
To this end, we rewrite the resolvent once more as
\begin{align}\label{eq:R0 R2}
\mR_0(z)(x,y)= \mR_{-1}(z)(x,y) e^{iz|x-y|}+\mR_2(z) e^{iz|x-y|}, \,\,\textrm{where }
\end{align}
$$
\mR_{-1}(z)(x,y):=\frac{i}2\big[\frac{m}{z}(\beta+I)-  \alpha \textrm{sgn}(x-y)\big],
$$
and using \eqref{resolventex} we see that $\mR_2(z)=c_1z+O_2(z^3)$. 
Using the additional factor of $z$ from $\mR_2$, by a minor variation of the proof of Proposition~\ref{prop:Q}, we have 
	\begin{align*}
	\bigg| \int_{\R^3} e^{-it\sqrt{z^2+m^2}+iz(|x-x_1|+|y-y_1|)} \frac{z\chi(z)}{\sqrt{z^2+m^2}} \mR_2  v^* (\frac{c_{-1}}{z} S_1-  S_1T P -PTS_1)v \mR_{-1} \,dz\,dx_1\,dy_1 \bigg|& \les  |t|^{-\frac32}\la x\ra \la y \ra ,\\
	\bigg| \int_{\R^3} e^{-it\sqrt{z^2+m^2}+iz(|x-x_1|+|y-y_1|)} \frac{z\chi(z)}{ \sqrt{z^2+m^2}} (\mR_{j}) v^* (\frac{c_{-1}}{z} S_1-  S_1T P-PTS_1 )v \mR_{2}\, dz \,dx_1\,dy_1\bigg|& \les |t|^{-\frac32}\la x\ra \la y \ra.
	\end{align*}
	Here we may select any $j\in \{-1,2\}$ for the second bound.
	It remains to consider 
\begin{multline*}
	\frac1{2\pi i \mathcal D}	\int_{\R^3} e^{-it\sqrt{z^2+m^2}+iz(|x-x_1|+|y-y_1|)} \frac{z\chi(z)}{\sqrt{z^2+m^2}} \mR_{-1}(z)(x,x_1)\\  [v^* \big[\frac{1}{c_Pz }S_1-S_1TP-PTS_1\big]v] (x_1,y_1) \mR_{-1}(z)(y_1,y)\, dz\, dx_1\,dy_1.
\end{multline*}
We start with $S_1$ and consider the most singular part  
	\begin{align}\label{eqn:S1 singular}
		\frac{-1}{8\pi i c_P\mathcal D}	 \int_{\R^3} e^{-it\sqrt{z^2+m^2}+iz(|x-x_1|+|y-y_1|)} \frac{\chi(z)}{z^2\sqrt{z^2+m^2}} \big[m(\beta+I)   v^*S_1v m(\beta+I)\big](x_1,y_1)\, dx_1 dy_1 dz. 
	\end{align}
	Using \eqref{eqn:Q} and $S_1\leq Q$, we have 
	$$ \eqref{eqn:S1 singular} = \frac{-1}{8\pi i c_P\mathcal D}\int_{\R^2} \int_{|x|}^{|x-x_1| }  \int_{|y|}^{|y-y_1| }  G_t(s_1+s_2)   \big[m(\beta+I)   v^*S_1v   m(\beta+I)\big](x_1,y_1)\,ds_1 ds_2 dx_1 dy_1, 
	$$
where $G_t(r)$ is as in \eqref{Gtdef} in the proof of Lemma~\ref{lem:Ft}. Furthermore, using \eqref{Gtexp} and 
	letting 
$$H_t(y_1,y) =\int_{|y|}^{|y-y_1| }  \chi(s_1/t) e^{-im\sqrt{t^2-s_1^2}}ds_1 ,$$ we have (up to an error term satisfying the weighted bound)
	$$
	\eqref{eqn:S1 singular}=\frac{-1}{8\pi i c_P\mathcal D}\frac{(-2\pi i)^{\frac12}e^{imt}}{(mt)^{\frac12}} \int_{\R^2} H_t(x_1,x)  \big[m (\beta+I)   v^*S_1v \, m (\beta+I)\big](x_1,y_1) H_t(y_1,y)\,  dx_1 dy_1. 
	$$ 
Similarly, if we consider the contribution of
	\begin{align*}
\frac{-1}{8\pi i c_P\mathcal D}	\int_{\R^3} e^{-it\sqrt{z^2+m^2}+iz(|x-x_1|+|y_1-y|)} \frac{\chi(z)}{z\sqrt{z^2+m^2}} \big[ -\alpha \textrm{sgn}(x-x_1)  v^*S_1v\,   m(\beta+I) \big]\, dx_1\, dy_1\,  dz
	\end{align*}
	we obtain (up to an error term satisfying the weighted bound)
	$$
	\frac{-1}{8\pi i  c_P\mathcal D} \frac{(-2\pi i)^{\frac12}e^{imt}}{(mt)^{\frac12}} \int_{\R^2} \chi(x/t)e^{-im\sqrt{t^2-x^2}}  (-\alpha) \textrm{sgn}(x-x_1)   [v^*S_1v\,   m(\beta+I)](x_1,y_1) H_t(y_1,y)\,  dx_1 dy_1. 
	$$ 
	Finally the contribution of 	
	\begin{align*}
\frac{-1}{8\pi i c_P\mathcal D}	\int_{\R^3} e^{-it\sqrt{z^2+m^2}+iz(|x-x_1|+|y_1-y|)} \frac{\chi(z)}{\sqrt{z^2+m^2}} \big[ -\alpha \, \textrm{sgn}(x-x_1)  v^*S_1v (-\alpha )\,\textrm{sgn}(y_1-y) \big]\, dx_1\, dy_1\,  dz
	\end{align*}
is (up to an error term satisfying the weighted bound)
	\begin{multline*}
		\frac{-1}{8\pi i c_P\mathcal D}	 \frac{(-2\pi i)^{\frac12}e^{imt}}{(mt)^{\frac12}} \int_{\R^2} \chi(x/t)e^{-im\sqrt{t^2-x^2}} (- \alpha )\, \textrm{sgn}(x-x_1)  \\ [v^*S_1v](x_1,y_1)(-\alpha) \,\textrm{sgn}(y_1-y) \chi(y/t)e^{-im\sqrt{t^2-y^2}}\, dx_1\, dy_1.
	\end{multline*}
	
	So that, the contribution of the $S_1$ term is
	\begin{multline}
\label{S1cont}	\frac{-1}{8\pi i c_P\mathcal D}\frac{(-2\pi i)^{\frac12}e^{imt}}{(mt)^{\frac12}} \int_{\R^2}  \big[-\chi(x/t)e^{-im\sqrt{t^2-x^2}} \alpha \,\textrm{sgn}(x-x_1)+ mH_t(x_1,x) (\beta+I) \big] (v^*S_1v) (x_1,y_1) \\
	\big[-\chi(y /t)e^{-im\sqrt{t^2-y^2}} \alpha \, \textrm{sgn}(y_1-y)+ mH_t(y_1,y)(\beta+I)\big]  \,  dx_1 dy_1
	+O(|t|^{-\f32}\la x\ra \la y\ra ). 
	\end{multline}
Picking a unit $\phi\in S_1	L^2$ as in Lemma~\ref{lem:S1 char}, note by Corollary~\ref{cor:S1} $S_1L^2$ is one dimensional, and defining 
    \begin{align*}
    	a_1(x,t)&:=\int_{\R}  \big[-\chi(x/t)e^{-im\sqrt{t^2-x^2}} \alpha \,\textrm{sgn}(x-x_1)+ mH_t(x_1,x) (\beta+I) \big] v^*(x_1)\phi(x_1)\, dx_1,\\
    	a_2(y,t)&:=\int_{\R} \phi^*(y_1) v(y_1)  \big[-\chi(y/t)e^{-im\sqrt{t^2-y^2}} \alpha \,\textrm{sgn}(y_1-y)+ mH_t(y_1,y) (\beta+I)  \big] \, dy_1,   
   \end{align*} 
we have
$$
\eqref{S1cont}	=\frac{-1}{8\pi i c_P\mathcal D}\frac{(-2\pi i)^{\frac12}e^{imt}}{(mt)^{\frac12}}  a_1(x,t)a_2(y,t)
	+O(|t|^{-\f32}\la x\ra \la y\ra ). 
	$$
Similarly, the contributions of $S_1TP$ and $PTS_1$ terms are
\begin{multline}
\label{S1TPcont} 
\frac{1}{8\pi i \mathcal D}\frac{(-2\pi i)^{\frac12}e^{imt}}{(mt)^{\frac12}} \int_{\R^2}  \big[-\chi(x/t)e^{-im\sqrt{t^2-x^2}} \alpha \,\textrm{sgn}(x-x_1)+ mH_t(x_1,x) (\beta+I) \big] (v^*S_1TPv) (x_1,y_1) \\
	\big[m\chi(y /t)e^{-im\sqrt{t^2-y^2}} (\beta+I)\big]  \,  dx_1 dy_1
	+O(|t|^{-\f32}\la x\ra \la y\ra ), \text{ and}
\end{multline}
\begin{multline}\label{PTS1cont}
\frac{1}{8\pi i \mathcal D}\frac{(-2\pi i)^{\frac12}e^{imt}}{(mt)^{\frac12}} \int_{\R^2}  \big[m\chi(x /t)e^{-im\sqrt{t^2-x^2}} (\beta+I)\big] (v^*PTS_1v) (x_1,y_1) \\
	\big[-\chi(y/t)e^{-im\sqrt{t^2-y^2}} \alpha \, \textrm{sgn}(y_1-y)+ mH_t(y_1,y) (\beta+I) \big]  \,  dx_1 dy_1
	+O(|t|^{-\f32}\la x\ra \la y\ra ).
\end{multline}
Using  unit  $\phi \in S_1L^2$  we have above,   we can write the kernel of operator $PTS_1$ as $PTS_1(x_1,y_1)=[PT\phi](x_1) \phi^*(y_1)$. From Lemma~\ref{lem:S1 char}, $PT\phi(x_1)=\kappa_0v(x_1)(1,1)^T $. Therefore
   	\begin{multline*}
   		\int_{\R} (\beta+I)v^*(x_1)PTS_1(x_1,y_1)\, dx_1=
     	\kappa_0  	\int_{\R} (\beta+I) v^*(x_1)  v(x_1)(1,1)^T  \phi^*(y_1)\, dx_1\\
     	=\kappa_0 \|(a+b,c+d)^T\|_2^2 (1,1)^T\phi^*(y_1) .
	\end{multline*}
Also using $a_2$ notation as above, we write
\begin{multline*}
\eqref{PTS1cont}=\frac{1}{8\pi i \mathcal D}\frac{(-2\pi i)^{\frac12}e^{imt}}{(mt)^{\frac12}}  \kappa_0 m \|(a+b,c+d)^T\|_2^2 \chi(x /t)e^{-im\sqrt{t^2-x^2}}  (1,1)^T a_2(y,t)+O(|t|^{-\f32}\la x\ra \la y\ra ) \\
= \frac{-1}{8\pi i c_P\mathcal D}\frac{(-2\pi i)^{\frac12}e^{imt}}{(mt)^{\frac12}}  b_1(x,t) a_2(y,t)+O(|t|^{-\f32}\la x\ra \la y\ra ),
\end{multline*}
where $b_1(x,t):=2i \kappa_0  \chi(x /t)e^{-im\sqrt{t^2-x^2}}  (1,1)^T$.  We used $-c_Pm \|(a+b,c+d)^T\|_2^2= 2i $ in the last equality.
Similarly,   
$$
\eqref{S1TPcont}=\frac{-1}{8\pi i c_P\mathcal D}\frac{(-2\pi i)^{\frac12}e^{imt}}{(mt)^{\frac12}}  a_1(x,t)b_2(y,t)+O(|t|^{-\f32}\la x\ra \la y\ra ),
$$
where $b_2(y,t):=2i \overline{\kappa_0}  \chi(y /t)e^{-im\sqrt{t^2-y^2}}    (1,1).$
Finally, noting that $\beta+I=(1,1)^T(1,1)$, we can write the contribution of $P$ term as
$$
\eqref{eqn:zP weird}=
	  \frac{-1}{8\pi i {\mathcal D} c_P}  \frac{(-2\pi i)^{\frac12}e^{imt}}{(mt)^{\frac12}} b_1(x,t) b_2(y,t)  +O(|t|^{-\frac32}\la x\ra \la y\ra).
$$
Note that we can express  the sum of the contributions of $P,S_1, S_1TP, PTS_1$ above as an operator with kernel of the form:
    \begin{align}\label{eqn:rank one}
    	\frac{-1}{8\pi i c_P\mathcal D}\frac{(-2\pi i)^{\frac12}e^{imt}}{(mt)^{\frac12}} \big[ a_1(x,t)+b_1(x,t)\big] \big[ a_2(y,t)+b_2(y,t)\big]+O(|t|^{-\frac32}\la x\ra \la y\ra)=:{F_t^+}+O(|t|^{-\frac32}\la x\ra \la y\ra).
    \end{align}
\end{proof}

Finally, we prove the last claim in Theorem~\ref{thm:main nonreg}.   

\begin{prop}\label{prop:Ft real}
	
	If the threshold energy $m$  is not regular and $|v(x)|\les \la x\ra^{-\f92-}$, then for $|t|>1$, there is an operator $ F_t^+(x,y)$ of rank one given by 
	$$
		 F_t^+(x,y)=\frac{1}{2\pi i c_P \mathcal D}\frac{(-2\pi i)^{\frac12}e^{-imt}}{(mt)^{\frac12}} \psi(x) [\psi(y)]^*,
	$$
	where $\mathcal D$ and $c_P$ are the constants from Remark~\ref{rmk:Binv leading}, so that
	$$
	\|e^{-itH}P_{ac}^+(H)\chi(H)- F_t^+\|_{L^1\to L^\infty} \les |t|^{-\f32} \la x\ra^2 \la y \ra^2.
	$$
	Furthermore, $\psi_+\in L^\infty$ is a canonical resonance function, a distributional solution to $H\psi_+=m\psi_+$.  Hence $\|F_t^+\|_{1\to \infty}\les |t|^{-\f12}$.
	
\end{prop}
A similar construction may be done for the negative threshold to obtain the operator $F_t^-$.  The rank at most two operator in the statement of Theorem~\ref{thm:main nonreg} is exactly ${F_t}=F_t^++F_t^-$.
\begin{proof}
Following the proof of Proposition~\ref{prop:nonreg}, we need only provide further detail in the construction of $F_t^+$.
Note that by Taylor expansion we have
$$
\chi(x/t)e^{-im\sqrt{t^2-x^2}}=e^{-imt}+O(\la x\ra^2/t), \textrm{ and }
$$ 
$$
H_t(x_1,x)=e^{-imt}(|x-x_1|-|x|)+O(\la x\ra^2\la x_1\ra^3/t).
$$
Inserting this into the functions $a_j, b_j$ in \eqref{eqn:rank one}, using $S_1\leq Q$ and \eqref{eqn:Q} we see 
\begin{align*}
	a_1(x,t)&=-e^{-imt}\int_{\R}  \big[ \alpha \,\textrm{sgn}(x-x_1)+  m(|x-x_1|-|x| ) (\beta+I) \big] v^*(x_1)\phi(x_1)\, dx_1 +O(\la x\ra^2\la x_1\ra^3/t)\\
	&=-e^{-imt}\int_{\R}  \big[ \alpha \,\textrm{sgn}(x-x_1)+  m|x-x_1| (\beta+I) \big] v^*(x_1)\phi(x_1)\, dx_1 +O(\la x\ra^2\la x_1\ra^3/t)\\
	&=2ie^{-imt}\int_{\R}  G_0(x,x_1) v^*(x_1)\phi(x_1)\, dx_1 +O(\la x\ra^2\la x_1\ra^3/t)
\end{align*} 
By Lemma~\ref{lem:S1 char}, $-G_0v^*\phi=\psi-\kappa_0(1,1)^T$.  So that
$$
	a_1(x,t)=2ie^{-imt}\big[ \psi(x)-\kappa_0(1,1)^T\big]+O(\la x\ra^2\la x_1\ra^3/t)
$$
Similarly, we see (up to the error term) that 
$a_2(y,t)= 2ie^{-imt}\big[ \psi(y)-\kappa_0(1,1)^T\big]^*$, 
$b_1(x,t)=2ie^{-imt} \kappa_0  (1,1)^T$, and 
$b_2(y,t)=2ie^{-imt} \big[\kappa_0  (1,1)^T\big]^*$.
Combining all these terms we obtain (up to the error term)
$$
	\eqref{eqn:rank one}=\frac{1}{2\pi i c_P \mathcal D}\frac{(-2\pi i)^{\frac12}e^{-imt}}{(mt)^{\frac12}} \psi(x) [\psi(y)]^*.
$$
\end{proof}

\begin{rmk}\label{rmk:neg branch}
	
	The analysis for the negative portion of the spectrum $(-\infty,-m]$ follows with minimal changes.  In the resolvent expansions in Section~\ref{sec:Minv} one uses the change of variables $\lambda=-\sqrt{m^2+z^2}$, while the projection $P$ will replaced by the projection operator $P_{-}$ with kernel
	$$
		P_{-}(x,y)=\|(a-b,c-d)^T\|_2^{-2} v(x)(\beta-I)v^*(y).
	$$
	
\end{rmk}

\section{Spectral subspaces associated to threshold obstructions}\label{sec:spec}

In this section we relate the subspace  $S_1L^2$ to distributional solutions of $H\psi=m\psi$ and the invertibility of operators that arise in the expansions for the spectral measure in Section~\ref{sec:Minv}. As usual we consider the positive threshold $\lambda=m$, the negative threshold analysis follows with minor modifications.  The calculations here follow the set-up established by Jensen and Nenciu in \cite{JN} for the one dimensional Schr\"odinger operator.

Recall that for $|v(x)|\les \la x\ra^{-\frac32-}$ regularity of $m$ is equivalent to the invertibility of $QTQ$.   We  relate the kernel, $S_1L^2$,  of  $QTQ$ to the distributional solutions of $H\psi=m\psi$ as follows. Let $\theta=\|(a+b,c+d)\|_2^{-1} v(1,1)^T$ be a canonical unit vector in $PL^2$.

\begin{lemma}\label{lem:S1 char}
	
	Assume $|v(x)|\les \la x\ra^{-2-}$. Then, if  $\phi \in S_1L^2(\R)\setminus\{0\}$, then $\phi=Uv\psi$ for some $\psi \in L^\infty(\R )\setminus\{0\}$ which is a distributional solution to $H\psi=m\psi$.  Further, $\psi=-G_0v^*\phi+\kappa_0(1,1)^T$, $\psi\in L^\infty(\R)$, with
	$$ \kappa_0=\la T\phi,v(1,1)^T\ra\|(a+b,c+d)\|_2^{-2}= \la T\phi,\theta\ra\| (a+b,c+d)\|_2^{-1}.$$ 
	Furthermore, $\psi$ cannot be an eigenvalue, i.e. $\psi\not\in L^2(\R)$.
\end{lemma}

\begin{proof}
	
	Recalling that $S_1\leq Q$, we can see that $P\phi=0$ and $\int (\beta+I)v^*(y) \phi(y)\, dy=0$.  So, for $\phi \in QL^2$ we have (with $\kappa_0$ as in the statement above)
	\begin{align*}
	0&= T\phi-PT\phi=T\phi-\kappa_0 v(1,1)^T =U\phi+vG_0v^* \phi-\kappa_0 v(1,1)^T.
	\end{align*}
	Letting  
	\begin{align}\label{eqn:psi G0}
	\psi=-G_0v^*\phi+\kappa_0 (1,1)^T,
	\end{align}
	and using $U^2=I$, we have $\phi=Uv\psi$. 
	We claim that $H\psi=m\psi$.  Namely, we consider
	\begin{align*}
	(H-mI)\psi &=(D_m-mI)\psi+V\psi=-(D_m-mI)G_0v^*\phi+v^*Uv\psi\\
	&=-(D_m-mI)G_0v^*\phi+v^*\phi.
	\end{align*}
	Here we used that $(D_m-mI)(1,1)^T=0$.  Our claim is proven provided   that $(D_m-mI)G_0v^*\phi=v^*\phi$.
	
	Noting \eqref{eq:G0 defn} we have
	\begin{align*}
		(D_m-mI)G_0v^*\phi=(D_m-mI)(D_m+mI)\bigg(\frac{-|x-y|}{2}\bigg) v^*\phi=(-\partial_{xx})\bigg(\frac{-|x-y|}{2}\bigg)v^*\phi.
	\end{align*}
	We may conclude this is equal to $v^*\phi$ provide $v^*\phi \in L^{1,1}$.

	We now show $\psi\in L^\infty$ as claimed.  Using \eqref{eqn:psi G0}, the constant vector is obviously bounded, we consider only the first portion.  In the sense of distributions, we have 
	\begin{align*}
	G_0v^*\phi&=(D_m+mI)\bigg(\frac{-|x-y|}{2}\bigg)v^*\phi\\
	&=-\frac{1}{4}(i\alpha \partial_x +m(\beta+I))\int_{\R} |x-y| v^*(y)\phi(y)\, dy\\
	&=-\frac{1}{4}(i\alpha \partial_x +m(\beta+I))\bigg[ \int_{-\infty}^x (x-y)v^{*}(y)\phi(y)\, dy+\int_x^{\infty} (y-x)v^{*}(y)\phi(y)\, dy \bigg]\\
	&=-\frac{i}{4}\alpha \bigg[\int_{-\infty}^x  v^{*}(y)\phi(y)\, dy-\int_x^{\infty}  v^{*}(y)\phi(y)\, dy  \bigg]- \frac{m}{4}  (\beta+I)\int_{\R} |x-y|v^*(y)\phi(y)\, dy.
	\end{align*}
	Here we note that since $\phi \in QL^2$, we must have that
	\begin{align*}
	-\frac{i}{4} \int_{\R} v^*(y)\phi(y)\, dy =\kappa_1 (1,-1)^T,
	\end{align*}
	for some constant $\kappa_1$.  This follows since
	$$
	(\beta+I)\int_{\R} v^*(y)\phi(y)\, dy= (0,0)^T.
	$$
	Further note that $ \alpha (1,-1)^T=-   (1,1)^T$.
	Using this, we write
	\begin{align*}
	\frac{i}{4}\alpha \bigg[\int_{-\infty}^x  v^{*}(y)\phi(y)\, dy-\int_x^{\infty}  v^{*}(y)\phi(y)\, dy  \bigg]
	&=\kappa_1 (1,1)^T  -\frac{i}{2}\alpha \int_x^\infty v^{*}(y)\phi(y)\, dy. 
	\end{align*}
	Similarly, if we denote the constant vector $  u=\int_{\R}y v^*(y)\phi(y)\, dy$ we have
	\begin{multline*}
	(\beta+I)\int_{\R} |x-y|v^*(y)\phi(y)\, dy\\
	= (\beta+I)\bigg[\int_{x}^\infty (y-x) v^*(y)\phi(y)\, dy+\int_{-\infty}^x (x-y) v^*(y)\phi(y)\, dy  \bigg]\\
	= (\beta+I) \big( 2\int_x^\infty (y-x)v^*(y)\phi(y)\, dy-u \big)
	\\=\kappa_2 (1,1)^T  + 2 (\beta+I)  
	\int_x^\infty (y-x)v^*(y)\phi(y)\, dy.
	\end{multline*}
	Combining all these facts, with some constant $\kappa$, we may write:
	\begin{align}\label{eqn:psi form}
	\psi=\kappa (1,1)^T + \frac{1}{2}   \int_x^\infty [-  m(y-x)(\beta+I)+i\alpha]v^*(y)\phi(y)\, dy  .
	\end{align} 
	This proves that $\psi $  is  bounded on $x>0$  provided that $v^*\phi \in L^{1,1}$ since $|y-x|\leq |y|$ on this domain.  A  similar argument proves that $\psi$ is also bounded on $x<0$ by writing  $\int_x^\infty=\int_\R-\int_{-\infty}^x $. 
 Now we prove that $\psi$ cannot be in $L^2\setminus \{0\}$.    
	Recalling that $\phi=Uv\psi$, we write 
	\begin{align*}
	\psi=\kappa(1,1)^T +\frac{1}{2} 
	\int_x^\infty [ -m(\beta+I)(y-x)+i\alpha]V(y)\psi(y)\, dy .
	\end{align*} 
	Note that if $\kappa=0$, then $\psi$ satisfies
	\begin{equation}\label{volter}
	\psi(x)=\int K(x,y) \psi(y) dy,
	\end{equation}
	where $|K(x,y)|\les  \la y-x\ra  |V(y)| \chi_{y>x}.$ By a Volterra integral argument this implies that $\psi\equiv 0$ provided that $|V(y)|\les \la y\ra^{-2-}$. Therefore $\psi\to \kappa(1,1)^T\neq 0$  as $x\to\infty$, and hence it cannot be in $L^2$.  
\end{proof}	
	
	We note that, denoting $PT\phi=\kappa_0v(1,1)^T$, $i\alpha \int v^*(y)\phi(y)\, dy=\kappa_1(1,1)^T$ and
	$u=(u_1,u_2)^T:=\int y  v^*(y)\phi(y)\, dy$ one obtains the expansion
	\begin{align}\label{eqn:psi for B1}
		\psi=\bigg[ \kappa_0-\frac{\kappa_1}{2}+\frac{m}{2}(\beta+I)u \bigg]  (1,1)^T +\frac{1}{2} \int_{-\infty}^x (i\alpha+m(\beta+I)(x-y))v^*(y)\phi(y)\, dy.
	\end{align}
	This follows by writing 
	\begin{multline*}
		\int_{\R} \textrm{sgn}(x-y)v^*(y)\phi(y)\, dy\\
		=	\int_{-\infty}^x  v^{*}(y)\phi(y)\, dy-\int_x^{\infty}  v^{*}(y)\phi(y)\, dy	=-\int_{\R}v^*(y)\phi(y)\,dy+2\int_{-\infty}^x v^*(y)\phi(y)\,dy,
	\end{multline*}
	with a similar expansion for the integral involving $|x-y|v^*(y)\phi(y)$ while noting that $(\beta+I)\int v^*(y)\phi(y)\, dy=0$.

	The analysis above gives the following structure for $\psi$: for some $\psi_1\in L^\infty$, some constants $c_1,c_2$, and $\epsilon>0$
	\begin{equation}\label{weirdpsi}
	\psi = (1,1)^T[c_1\widetilde\chi_{x>0}+c_2\widetilde\chi_{x<0}]+\la x\ra^{-\epsilon} \psi_1(x).
	\end{equation} 

	\begin{lemma} If    $\psi\neq 0$ satisfying \eqref{weirdpsi}  is a distributional solution of $H\psi=m\psi$, then $ \phi=Uv\psi\in S_1L^2$,  the kernel  of $QTQ$.
	\end{lemma}
	\begin{proof}
	 First we need to see that $P\phi=0$, i.e. we need to establish
	$$
	(\beta+I) \int (v^* Uv)(y) \psi(y)dy=(\beta+I)\int V(y) \psi(y) dy =0.
	$$
	Using $(\beta+I)(\beta-I) =0$, $(\beta+I)\alpha (1,1)^T=0$, \eqref{weirdpsi}, and noting that 
	$$0=(H-mI)\psi=  [i\alpha \partial_x +m(\beta-I)]\psi +V\psi,$$
	 we conclude that 
	$$
	(\beta+I)V(x)\psi(x) = - i(\beta+I)\alpha\partial_x  (\la x\ra^{-\epsilon}\psi_1(x)). 
	$$ 
	Therefore for any test function $g$ with $g(0)=1$, we have 
	$$
	\int (\beta+I)V(y)\psi(y) g(\delta y) dy =  i(\beta+I)\alpha \int \la y\ra^{-\epsilon} \psi_1(y) \delta g^\prime(\delta y) dy=i(\beta+I)\alpha \int \la y/\delta \ra^{-\epsilon} \psi_1(y/\delta)  g^\prime(  y) dy \to 0,
	$$
	as $\delta \to 0^+$. Therefore by Lebesgue dominated convergence theorem, we conclude that $
	\int (\beta+I)V(y)\psi(y) dy =0$ as needed. 
	
	Now,
	$$QTQ \phi=Q T\phi=Q(U+vG_0v^*)Uv\psi= Qv(\psi+G_0V\psi).$$
	Note that $(D_m-mI)(\psi+G_0V\psi) = -V\psi +(D_m-mI)G_0V\psi=0$ since $V\psi\in L^{1,1}$. Therefore $g := \psi+G_0V\psi$ is a  distributional solution of $i\alpha \partial_x g+m(\beta-I) g =0$. Since $[\alpha(\beta-I)]^2=0$, this implies that   $g(x)=[I+im\alpha(\beta-I) x]u$  for some constant vector $u$. Note that, by the proof of previous lemma, $G_0V\psi=G_0v^*\phi$ is bounded, and hence  $g$ is bounded. We conclude that  $g=u= c(1,1)^T  $, and hence $Qv g=0$. 
	
\end{proof}

\begin{corollary}\label{cor:S1}
	
	Assume $|v(x)|\les \la x\ra^{-2-}$. Then, $S_1L^2(\R)$ is at most one dimensional.
	
\end{corollary}	

\begin{proof}	
	
	To see that the resonance space is at most one dimensional, take $\psi_1$, $\psi_2$ as in \eqref{eqn:psi form}. We see that $\psi_1 +d\psi_2$  satisfies \eqref{volter} for some $d\neq 0$. Therefore, it vanishes by the Volterra argument. This implies that the resonance space is at most one dimensional. 
	
\end{proof}

\begin{lemma}\label{prop:B1 inv}
	
	If $|V(x)|\les \la x \ra^{-2-}$ and $S_1\neq 0$, then
	$$
	 |\kappa_0|^2 - \frac{im}{2}  \text{Tr}(S_1M_1S_1)  \neq 0.
	$$ 
	Consequently, we have $hk-|\ell|^2\neq 0$ in \eqref{eq:det}.
	
\end{lemma}

\begin{proof}
	
	By Corollary~\ref{cor:S1}, $S_1L^2$ is a one dimensional subspace.  We showed that if $\phi\in S_1L^2$ then $\psi =-G_0v^*\phi +\kappa_0(1,1)^T$ is a solution to $H\psi=m\psi$.  Recall \eqref{eqn:psi for B1}, with 
	$PT\phi=\kappa_0v(1,1)^T$, $i\alpha \int v^*(y)\phi(y)\, dy=\kappa_1(1,1)^T$ and
	$u=(u_1,u_2)^T:=\int y  v^*(y)\phi(y)\, dy$. From the proof of Lemma~\ref{lem:S1 char}, by the Volterra integral argument
	under the decay assumptions on $V$, we must have that $\kappa_0-\frac{\kappa_1}{2}+\frac{m}{2}(u_1+u_2)\neq 0$ or $\psi \equiv 0$ and $S_1=0$.

	We first  consider the contribution of Tr$(S_1M_1S_1)$.  Recalling that $M_1=vG_1v^*$ and \eqref{eq:G1 defn}, 
	$$
	G_1(x,y) = \frac{1}2  \alpha  (x-y)+ \frac{-im}{4} (\beta+I)|x-y|^2+\frac{i}{4m}I,
	$$
	we look at 
	\begin{multline}\label{eqn:B1 c1}
	\frac12 \int \phi^*(x)v(x)  \alpha (x-y) v^*(y)\phi(y)\, dy \, dx\\
	=\frac{-i}2\bigg( \int x\phi^*(x)v(x)\, dx  \bigg) \bigg( \int i\alpha v^{*}(y)\phi(y)\, dy \bigg)+\frac{-i}2 \bigg(- i \alpha \int x\phi^*(x)v(x)\, dx  \bigg) \bigg( \int   v^{*}(y)\phi(y)\, dy \bigg)\\
	=\frac{-i}2u^* \kappa_1 (1,1)^T+\frac{-i}2 \big[\kappa_1 (1,1)^T\big]^*u=\frac{-i}2\kappa_1(\overline{u_1+u_2})+\frac{-i}2\overline{\kappa_1}(u_1+u_2).
	\end{multline}
	Now, noting that $S_1\leq Q$, using \eqref{eqn:Q} we have 
	$$
	\int x^2\phi^*(x)v(x)(\beta+I)v^*(y)\phi(y)\, dy\, dx=0.
	$$
	Writing $|x-y|^2=x^2-2xy+y^2$ we see that
	\begin{multline}\label{eqn:B1 xy2}
	\frac{-im}{4} \int \phi^*(x)v(x)(\beta+I)|x-y|^2v^*(y)\phi(y)\, dy\, dx\\
	=\frac{ im}{2} \int x\phi^*(x)v(x)(\beta+I)y v^*(y)\phi(y)\, dy\, dx
	=\frac{ im}{2}u^* (\beta+I) u=\frac{ im}{2}|u_1+u_2|^2.
	\end{multline}
	Finally, we have
	\begin{align}\label{eqn:B1 Id}
	\frac{i}{4m}\int \phi^*(x)v(x)I v^*(y)\phi(y)\, dy\, dx
	=\frac{i}{4m}\big( i\kappa_1 (1,1)^T \big)^* \big( i\kappa_1 (1,1)^T\big) =\frac{i}{2m} |\kappa_1|^2.
	\end{align} 
	Now, combining \eqref{eqn:B1 c1}, \eqref{eqn:B1 xy2} \eqref{eqn:B1 Id}, we see that
	\begin{multline*}
	|\kappa_0|^2-\frac{im}{2}\text{Tr} (S_1M_1S_1) 
	=  |\kappa_0|^2 -\frac{m\kappa_1}{4}(\overline{u_1+u_2})-\frac{m\overline{\kappa_1}}{4} (u_1+u_2)+\frac{m^2}{4}|u_1+u_2|^2+\frac{1}{4}|\kappa_1|^2 \\
	= |\kappa_0|^2+ \bigg|-\frac{\kappa_1}{2}+\frac{m}{2}(u_1+u_2)  \bigg|^2 .
	\end{multline*}
	Now, for this to be zero we need  both $\kappa_0=0$ and $-\frac{\kappa_1}{2}+\frac{m}{2}(u_1+u_2)= 0$, hence implying that $\kappa_0-\frac{\kappa_1}{2}+\frac{m}{2}(u_1+u_2)=0$, which is a contradiction. 
\end{proof}

Finally, as in \cite{JN},  we construct an example of a resonance and eigenfunction with explicit potential hence demonstrating that such obstructions exist for the type of potentials considered here.

\begin{example} The function $
	\psi_+(x)=   e^{-\la x \ra^{\delta}}(1,1)^T $
	with self-adjoint potential
	$$
	V(x)=\begin{pmatrix}
	0 & i \delta x\la x\ra^{\delta-2}\\ -i \delta x\la x\ra^{\delta-2} & 0
	\end{pmatrix}
	$$
	is a distributional solution to $H\psi_+=m\psi_+$.  Hence $\psi_+ $ is a  resonance if $\delta<0$, $\psi_+ \to (1,1)^T$ as $|x|\to \infty$ and $|V(x)|\les \la x\ra^{\delta-1}$.  If $0<\delta<1$, $\psi_+$ is an eigenfunction.
	
	Similarly   $ \psi_-(x)= e^{-\la x \ra^{\delta}}(-1,1)^T $
	with self-adjoint potential
	$$
	V(x)=\begin{pmatrix}
	0 & -i \delta x\la x\ra^{\delta-2}\\  i \delta x\la x\ra^{\delta-2} & 0
	\end{pmatrix}
	$$
	is a distributional solution to $H\psi_-=-m\psi_-$.
	
\end{example}

\section{Limiting absorption principle}\label{sec:LAP}
In this section we prove Theorem~\ref{cor:uniform LAP} to obtain a limiting absorption principle that is uniformly bounded on the continuous spectrum.
We begin with energies close to the threshold $m$. We only work with $\mathcal R_0^+$ and drop the $\pm$ signs as usual.  Using the expansions and tools developed in Sections~\ref{sec:Minv} and \ref{sec:low disp}, one can easily obtain a limiting absorption principle for $\sigma>\f32$.  To obtain the sharper bound of Theorem~\ref{cor:uniform LAP}, we modify the argument to first prove
\begin{lemma}\label{lem:LAPlow}
Assume that $|V(x)|\les \la x\ra^{-3-}$, and that $m$ is a regular point of the spectrum. Then
for all $\sigma>1$ we have
$$
\sup_{0<|z|<z_0} \|\mathcal R_V(z)\|_{L^{2,\sigma}\to L^{2,-\sigma}}\les 1.
$$
\end{lemma} 
\begin{proof}
Using \eqref{resolventex} we have
$$
      \mathcal{R}_0 (z)(x,y)=   
      \frac{im}{2z}   (\beta+I)  e^{iz|x-y|} +O(1). 
$$  
Using this and the first expansion in Proposition~\ref{Minversenew} in the symmetric resolvent identity \eqref{eqn:symm res id}, we obtain
\begin{multline*}
\mathcal R_V(z)(x,y)=\frac{im}{2z}   (\beta+I)  e^{iz|x-y|} +O(1)\\
-\int_{\R^2}\big[\frac{im}{2z}   (\beta+I)  e^{iz|x-x_1|} +O(1)\big] \bigg[v^*\big[c_P z P + z^2\Lambda_0(z)+z \Lambda_1(z)Q+z Q\Lambda_2(z)+Q\Lambda_3(z)Q\big]v\bigg](x_1,y_1)\\
\big[\frac{im}{2z}   (\beta+I)  e^{iz|y-y_1|} +O(1)\big] dx_1 dy_1,
\end{multline*}
provided that 
  $m$ is a regular point of the spectrum and $|v(x)|\les \la x\ra^{-\tfrac12-}$.  Here each $\Lambda_j$ is unifomly bounded in $L^2$ for $0<|z|<z_0$. 

Note that  we need to prove that the operator with kernel $\mathcal R_V(z)(x,y) \la x\ra^{-1-}\la y\ra^{-1-}$ is uniformly bounded in $L^2$ for $0<|z|<z_0$. This follows by using the orthogonality as in \eqref{eqn:Q} and \eqref{eqn:F defn} to eliminate the factors of $\frac{1}{z}$ for all terms  except  for the leading singular terms.    It remains only to control the following terms:
%SHOOT! I MADE ALL THE MODIFICATIONS NOT TO NEED THIS MUCH DECAY. DANG. I CAN CONVERT BACK TO AN EARLIER FILE IF YOU DON''T LIKE THE CHANGES. 
\begin{multline*}
 \frac{im}{2z}   (\beta+I)  e^{iz|x-y|} 
+\frac{m^2c_P}{4z }\int_{\R^2}     e^{iz|x-x_1|}   \big[(\beta+I) v^*   P  v(\beta+I) \big](x_1,y_1)   e^{iz|x-x_1|} dx_1 dy_1\\
= \frac{im}{2z}   (\beta+I)  \big[e^{iz|x-y|} - e^{iz(|x|+|y|)}\big]+O(1) \\
=O\big(||x-y|-|x|-|y||+1\big) =
O\big(\min(\la x\ra,\la y\ra\big).
\end{multline*}
Where we used $e^{iz|x-x_1|}=e^{iz|x|}+O(|z||x_1|)$, similarly $e^{iz|y-y_1|}=e^{iz|y|}+O(|z||y_1|)$,  $c_P=\frac{-2i    }{m \| (a+b,c+d)^T\|_2^{ 2} }$ and \eqref{eqn:S P int}.  Controlling the error term here, and utilizing \eqref{eqn:F defn} necessitate the assumption that $|V(x)|\les \la x\ra^{-3-}$.
Now note that, by Schur's test, the operator with kernel 
$$
\frac{\min(\la x\ra,\la y\ra)}{\la x\ra^{\sigma}  \la y\ra^{\sigma}} , \qquad\sigma>1
$$
is bounded in $L^2(\R)$, thus establishing the claim. 
\end{proof}

We now consider energies away from the threshold. 
First note that the free resolvent
\begin{multline}
\label{resolventex1}
\mathcal{R}_0 (z)(x,y)= \big[  i \alpha \partial_x   + m \beta + \sqrt{m^2+ z^2} I \big] \frac{ie^{iz|x-y|}}{2z} = \\ 
\frac{i}2\big[ -  \alpha \,\textrm{sgn}(x-y) + \frac{m\beta+\sqrt{z^2+m^2}I}{z}  \big] e^{iz|x-y|}=: f_0(z,x,y)e^{iz|x-y|},
\end{multline}
trivially  satisfies for all $|z|>z_0>0$ and $\sigma>\frac12$ 
\be\label{freeLAP}
\|\mathcal{R}_0 (z)\|_{L^{2,\sigma} \to L^{2,-\sigma}}\leq C_{\sigma,z_0}.
\ee
We also have  
\be \label{eq:res_iden}
\mR_V  (z)= \mR_0 (z )\big[I + V\mR_0 (z)\big]^{-1}. 
\ee
We can write for fixed  $d>0$ 
\begin{align}\label{eqn:Rd1}
\mR_0= \mR_d^1+\mR_d^2+\mR_d^3,
\end{align}
where 
\begin{align*}
	\mR_d^1(z)(x,y)&= \frac{i}2\big[ -  \alpha  + \frac{m\beta+\sqrt{z^2+m^2}I}{z}  \big]  e^{ i  z(x-y)}\chi_{(d,\infty)}(x-y) =f_1(z)e^{ i  z(x-y)}\chi_{(d,\infty)}(x-y),\\
	\mR_d^2(z )(x,y)&=\frac{i}2\big[ +  \alpha  + \frac{m\beta+\sqrt{z^2+m^2}I}{z}  \big]   e^{ i  z(y-x)}\chi_{(d,\infty)}(y-x)  =f_2(z)e^{ i  z(y-x)}\chi_{(d,\infty)}(y-x) , \text{ and} \\
	\mR_d^3(z)(x,y)&=\frac{i}2\big[ -  \alpha \,\textrm{sgn}(x-y) + \frac{m\beta+\sqrt{z^2+m^2}I}{z}  \big] e^{ i  z|x-y|}\chi_{(-2d,2d)}(x-y)=: f_3(z,x-y) e^{ i  z|x-y|}, 
\end{align*}
where $\chi_I$ denotes a smooth cutoff supported in the interval $I$ so that $\chi_{(d,\infty)}(x)+\chi_{(d,\infty)}(-x)+\chi_{(-2d,2d)}(x)\equiv 1$. 

Note that $\mR_d^1$ and $\mR_d^2$ satisfy \eqref{freeLAP}, and by Schur's test, we have  
$$
\|\mR_d^3\|_{L^2 \to L^2 } \les d.
$$ 
Using these bounds we have 
$$\|V\mR_d^3\|_{L^{2,\sigma} \to L^{2,-\sigma}}\leq  d C_V, \text{ and } 
$$
$$\|V\mR_d^1\|_{L^{2,\sigma} \to L^{2,-\sigma}}, \|V\mR_d^2\|_{L^{2,\sigma} \to L^{2,-\sigma}}  \leq C_{\sigma,V},\,\,\text{ for }\sigma>\frac12,\, |z|>z_0>0
$$
provided that $|V(x)|\les \la x\ra^{-2\sigma}$. Also note that all of the bounds above hold when we replace the kernel with its absolute value. 

These and a variant of the argument in the next section controlling the infinite series expansion \eqref{eqn:hi series} establish the limiting absorption principle for large energies  as in \cite{EGG} (also see \cite{EGS1,EGS2}).  Combining this with the results of Georgescu and Mantoiu for compact subsets of the spectrum, \cite{GM}, we have the following uniform limiting absorption principle:
\begin{lemma}\label{lem:LAP} Assume that $V$ has continuous entries. Then
	for any $|z|>z_0>0$, and $k=0,1,2,\ldots, $ we have
	$$
	\|\partial_z^k \mR_V^\pm(z)\|_{L^{2,\sigma} \to L^{2,-\sigma}} \leq C_{\sigma,k,V},
	$$ 
	provided that $\sigma>\frac12+k$, and  $|V(x)| \les \la x\ra^{-2\sigma-}$.
\end{lemma}

For completeness, we discuss the consequences of the uniform limiting absorption principle in Theorem~\ref{cor:uniform LAP}.  An immediate consequence, see
e.g.~\cite[Theorems XIII.25 and XIII.30]{RS1}, is the Kato smoothing bound:
$$
\|\la x\ra^{-\sigma}e^{-itH}f\|_{L^2_tL^2_x}\leq C_{\sigma,H} \|f\|_{L^2}, 
$$
provided that the thresholds are regular, $\sigma>1$,  $|V(x)|\les \la x\ra^{-3-}$, and $V$ has continuous entries.
From here, following the argument of Rodnianski and Schlag in \cite{RS}, one obtains the Strichartz estimates in Corollary~\ref{cor:Strichartz}, also see \cite{EGS1,EGS2} and Section~2 of \cite{EGG}.  The eigenvalue free region follows from a standard perturbation argument.
We refer the reader to Section~6 of \cite{EGG} for a more complete discussion.

\section{High energy dispersive estimates}\label{sec:high}

To complete the proof of Theorems~\ref{thm:main reg}, we consider the contribution of the Stone's formula \eqref{eqn:stone} when the spectral variable is bounded away from the threshold energies. We begin by considering the case when $|z|$ is sufficiently large. We adapt the argument of \cite{EGG}, which was designed to establish  large energy limiting absorption principles, to obtain $L^1\to L^\infty $ bounds with almost optimal derivative requirement on the initial data. We note that this method has not been used before to obtain global decay estimates. The reason that it works in the case of one dimensional Dirac equation is that,  unlike in dimensions $d\geq 2$, the magnitude of the free resolvent on $\R$ also satisfies the limiting absorption principle. However, the oscillation in $\mR_0$ is still crucial to establish the limiting absorption principle for the perturbed resolvent and  for dispersive estimates.

\begin{prop}\label{prop:hi energies}
	Assume $|V(x)|\les \la x\ra^{-\delta}$ and $|\partial_x V(x)|\les \la x\ra^{-1-}$.  If $\chi_j(z)$ is a smooth, even cut-off to frequencies $|z|\approx 2^j$ and $\delta>3$,
	\begin{align*}
	\bigg|\int_{\R} e^{-it\sqrt{z^2+m^2}} \frac{z\chi_j(z)}{\sqrt{z^2+m^2}} [\mR_V(z)](x,y)\, dz\bigg| \les \min(2^j, |t|^{-\f12}2^{\frac{3j}{2}} ).
	\end{align*}
	Furthermore, if $\delta>5$, and $|\partial_x V(x)|\les \la x\ra^{-2-}$
	\begin{align*}
	\bigg|\int_{\R} e^{-it\sqrt{z^2+m^2}} \frac{z\chi_j(z)}{\sqrt{z^2+m^2}} [\mR_V(z)](x,y)\, dz\bigg| \les \min(2^j, |t|^{-\f12}2^{\frac{3j}{2}},|t|^{-\frac32} 2^{\frac{3j}{2}} \max(\la x\ra, \la y\ra )   ).
	\end{align*}
\end{prop}

To establish the desired bounds, we utilize a different approach depending on the size of $|z|$.
To establish the desired bounds from Proposition~\ref{prop:hi energies} for sufficiently large $|z|$, we iterate the resolvent identity to form an infinite series.  Formally, we write 
\begin{align}\label{eqn:hi series}
	\mR_V(z)=\sum_{m=0}^{\infty} \mR_0(z)[-V\mR_0(z)]^m.
\end{align}
By the proof of the limiting absorption principle, the series converges in operator norm $:L^{2,\sigma}\to L^{2,-\sigma}$, $\sigma>\frac12$, for large enough $|z|$.  To study the contribution of the $m^{th}$ term of the series to the dispersive estimate, one cannot utilize decay in the spectral variable, as can be done for the Schr\"odinger operator \cite{GS}.  Instead we rely on a subtle cancellation that we exploit below for sufficiently large $|z|$.      When $|z|$ is bounded above, a simpler argument may be employed see Proposition~\ref{prop:int} and Lemma~\ref{lem:RV int} below.

For the dispersive bounds, we select a term from the series \eqref{eqn:hi series}, and show that contribution to the dispersive estimate is summable in $m$.  To that end,  consider for some $M$ to be determined and $0\leq r<M$ and write $m=kM+r$,
\be\label{highen_1}
\int_\R e^{-it\sqrt{z^2+m^2}} \frac{z}{\sqrt{z^2+m^2}} \chi_j(z) \big[\mR_0(z)   [V \mR_0(z)]^{kM} [V\mR_0(z)]^r  \big](x,y) dz.
\ee
Recalling \eqref{eqn:Rd1}, we
write $\mR_0=\mR_d^1+\mR_d^2+\mR_d^3$.  For each block of 
$$V\mR_d^{j_1}\cdots V \mR_d^{j_M},$$ 
where $j_i\in\{1,2,3\}$ we have two choices; after removing all instances of $j=3$, either all remaining $j$'s are the same, or there is a pair of $j$'s taking different values $1,2$. As in \cite{EGS1,EGS2,EGG}, we call the first case a directed product, and the latter case an undirected product. A product consisting of all $\mR_d^3$'s is a directed product.

We have the following lemma for directed products. We omit  the elementary proof which is similar to the ones given in  \cite{EGS2,EGG}.  
\begin{lemma} \label{lem:directedprod} Fix a potential $V$ satisfying $|V(x)|\les \la x\ra^{-1-}$.
	For any $\delta>0$, there exists a distance $d=d(\delta)> 0$ such that
	each directed product   satisfies the estimate
	\begin{equation}
	\big  |[V\mR_d^{j_1}\cdots V \mR_d^{j_M}](x,y) \big|\leq  |V(x)|\, C_{V,\delta}\, \delta^{M} 
	\end{equation}
	uniformly over all $z> 1$ and all choices of $M$. Moreover, the same claim holds if we replace the kernel of $\mR_d^j$'s with their absolute value.  
\end{lemma}  

We note that this bound does not rely on the oscillation  of the resolvent, and it follows because of the support condition on the kernels.
We are now ready to prove Proposition~\ref{prop:hi energies} for $|z|$ sufficiently large.

\begin{lemma}\label{lem:hi energies}
	
	Assume $|V(x)|\les \la x\ra^{-\delta}$ and $|\partial_x V(x)|\les \la x\ra^{-1-}$.  If $\delta>1$, there exists a $J$ sufficiently large so that for all $j\geq J$, we have
	\begin{align*}
	\bigg|\int_{\R} e^{-it\sqrt{z^2+m^2}} \frac{z\chi_j(z)}{\sqrt{z^2+m^2}} [\mR_V(z)](x,y)\, dz\bigg| \les \min(2^j, |t|^{-\f12}2^{\frac{3j}{2}} ).
	\end{align*}
	Furthermore, if $\delta>2$, and $|\partial_x V(x)|\les \la x\ra^{-2-}$
	\begin{align*}
	\bigg|\int_{\R} e^{-it\sqrt{z^2+m^2}} \frac{z\chi_j(z)}{\sqrt{z^2+m^2}} [\mR_V(z)](x,y)\, dz\bigg| \les \min(2^j, |t|^{-\f12}2^{\frac{3j}{2}},|t|^{-\frac32} 2^{\frac{3j}{2}} \max(\la x\ra, \la y\ra )   ).
	\end{align*}
	
\end{lemma}

\begin{proof}
From the preceding discussion of directed products in Lemma~\ref{lem:directedprod}, it suffices to consider undirected products.  For an arbitrary undirected product, the block contains a factor of the form 
$$
\mR_d^1V(\mR_d^3V)^n\mR_d^2, \text{ or of the form } \mR_d^2V(\mR_d^3V)^n\mR_d^1
$$
for some $n\geq 0$. Therefore, the integral in some adjacent space variables is either of the form 
\begin{multline}\label{undirectibp}
\int_{\R^{n+1}}  e^{iz(x_0-x_1+x_{n+2}-x_{n+1}+\sum_{\ell=1}^n|x_\ell-x_{\ell+1}|)} f_1\chi_{(d,\infty)}(x_0-x_1)f_2\chi_{(d,\infty)}(x_{n+2}-x_{n+1})V(x_1) \times \\
\times \Big[\prod_{\ell=1}^{n} f_3(x_\ell-x_{\ell+1})  V(x_{\ell+1})\Big]  dx_1\ldots dx_{n+1},
\end{multline}
or a similar one with an harmless sign change in the phase.   Recall \eqref{resolventex1} and \eqref{eqn:Rd1} for the definitions of $f_j$.
Letting $\frac{x_\ell-x_{\ell+1}}2=u_\ell$, $\ell=1,\ldots,n$, and let $\frac{x_1+x_{n+1}}2=u_0$. We have (with $U:=\sum_{\ell=1}^n u_\ell$)
$$
x_1=U+u_0, x_2=U+u_0-2u_1, \ldots , x_{n+1}=U+u_0-2u_1-\cdots -2u_n= u_0-U.
$$  
We therefore rewrite the integral as
\begin{multline*}
\int_{\R^{n+1}}  e^{iz(x_0+x_{n+2} +\sum_{\ell=1}^n|u_\ell|)} e^{-2izu_0} f_1\chi_{(d,\infty)}(x_0-u_0-U)f_2\chi_{(d,\infty)}(x_{n+2}-u_0+U)V(u_0+U) \times \\
\times \Big[\prod_{\ell=1}^{n} f_3(u_\ell)  V(x_{\ell+1})\Big]  du_0 \ldots du_{n}.
\end{multline*}
Integrating by parts in $u_0$ variable, we obtain
\begin{multline*} =-\frac{i}{2z}  \int_{\R^{n+1}} e^{iz(x_0+x_{n+2} +\sum_{\ell=1}^n|u_\ell|)} e^{-2izu_0} f_1f_2 \big[\prod_{\ell=1}^{n} f_3(u_\ell)\big] \times \\
\times \partial_{u_0}\Big[ \chi_{(d,\infty)}(x_0-x_1) \chi_{(d,\infty)}(x_{n+2}-x_{n+1})V(x_1)  
\prod_{\ell=1}^{n}   V(x_{\ell+1})\Big]  du_0 \ldots du_{n}.
\end{multline*}
Note that $\partial_{u_0} x_{\ell} =\pm 1$, $\ell=1,\ldots,n+1$. Therefore, up to a constant, the integral in \eqref{undirectibp}
is equal to a sum of integrals each of the form $\frac1z$ times the integral in \eqref{undirectibp} with one of the potentials  or one of the cutoff functions $\chi_{(d,\infty)} $  replaced with its derivative.

We repeat this for each undirected $M$-block. After this we rewrite  the integral in  \eqref{highen_1} by removing all oscillatory factors from the product (where $x_0=x$, $x_{kM+r+1}=y$):
$$
\int_{\R^{kM+r+1}} e^{-it\sqrt{z^2+m^2}+iz\sum_{i=0}^{kM+r}|x_i-x_{i+1}|} \frac{z}{\sqrt{z^2+m^2}} \chi_j(z)  f_0 O_1\cdots O_k  [Vf_0]^r   dz dx_1\cdots dx_{kM+r}.
$$ 
Here each directed block $O_\ell$ consists of an $M$-fold  product of    $Vf_1\chi_{(d,\infty)}$ and $Vf_3$ or a product of  $f_2$'s and $f_3$'s. And each undirected block consists of a product of $Vf_*\chi_{(d,\infty)}$, with $f_*=f_1, f_2, f_3$, and one potential or a cutoff $\chi_{(d,\infty)}$  is replaced with $\frac{C}{z}$ times its derivative.   Note that there are $3^{Mk}$ choices for $O_1,\ldots O_k$.

By Lemma~\ref{lem:vdc}, we  bound the $z$ integral by 
$$
\frac{2^{\frac{3j}{2}}}{|t|^{\f12}}  \int_{|z|\approx 2^{j}} \int \big|\partial_z \big( \frac{z \chi_j(z)}{\sqrt{z^2+m^2}} f_0 O_1\cdots O_k  [Vf_0]^r \big) \big| dx_1\cdots dx_{kM+r} dz.
$$
Note that wherever the $z$ derivative hits, we gain a factor of $2^{-j}$. Therefore, we consider the case when the derivative hits the first term, and estimate the others similarly by product rule:  
$$
\les \frac{2^{\frac{3j}{2}}}{|t|^{\f12}} (MK+r+1)  \int \big|   f_0 O_1\cdots O_k  [Vf_0]^r   \big| dx_1\cdots dx_{kM+r},
$$
where $f_0$, $f_1$, $f_2$, $f_3$ are replaced with their supremum on $|z|\approx 2^j$. They are all of order 1 (uniformly in $j$), except one factor in each undirected product which is of order $2^{-j}$. 

Note that for an undirected product $|O_i(x_{\ell},x_{\ell+M})|$ can be bounded by a sum of $\les M$ terms of  the form
$$
2^{-j} C^M |V(x_{\ell})| \int \big( \prod_{n=\ell+1,n\neq i}^{\ell+M-1}|V(x_{n})| \big) 
\big|\partial_{x_i}  V(x_i) \big|   dx_{\ell+1}\ldots dx_{\ell+M-1} \leq 2^{-j} C_V^M |V(x_{\ell})|.
$$
or of the form 
$$
2^{-j} C^M |V(x_{\ell})| \int \big( \prod_{n=\ell+1}^{\ell+M-1}|V(x_{n})| \big) 
\big|\partial_{x_i}  \chi_{(d,\infty)}(x_{i-1}-x_i) \big|   dx_{\ell+1}\ldots dx_{\ell+M-1} 
\leq 2^{-j}  C_V^M |V(x_{\ell})|.
$$
The term when the derivative hits the cutoffs is harmless since $\|\chi_{(d,\infty)}^\prime\|_{L^1}\les 1$ uniformly in $d>0$.
Therefore, for undirected products we have 
$$|O_i(x_{\ell},x_{\ell+M})|\leq 2^{-j} C_V^M |V(x_{\ell})|,$$
uniformly in $j$ and $d>0$. 

For  directed products, using Lemma~\ref{lem:directedprod},  given $\delta>0$,   there exists  $ d(\delta) >0$  so that  we have 
$$
|O_i(x_{\ell},x_{\ell+M})|\leq C_{V,\delta } \delta^{ M} |V(x_\ell)|
$$
for all $M$. Finally we have 
$$|[Vf_0]^r(x_{\ell},x_{\ell+r})|\leq C_V^r|V(x_\ell)|.$$
Combining all these inequalities, and renaming the variables, we have the bound 
\begin{multline*}
\les \frac{2^{\frac{3j}{2}}}{|t|^{\f12}}  3^{Mk}  C_{V}^{ r}  \max\big( C_V^M2^{-j},C_{V,\delta}\delta^{ M}\big)^k  \int \big| V(x_1)\ldots  V(x_{k+1})   \big| dx_1\cdots dx_{k+1}\\
\les \frac{2^{\frac{3j}{2}}}{|t|^{\f12}}   C_{V}^{ M}  \max\big( (3C_V^2)^M2^{-j},C_{V,\delta}C_V (3\delta)^{ M}\big)^k.  
\end{multline*}
Recall from \eqref{highen_1} that $m=kM+r$ with $0\leq r<M$.
This is summable in $m$ by first choosing $\delta$ small, then $M$ large, then  $2^j$ large.

The weighted bound $ |t|^{-\frac32} 2^{\frac{3j}2} \la x\ra \la y\ra$ follows the same way by Lemma~\ref{lem:vdc} provided that 
$|V(x)|,|\partial_x V(x)| \les \la x\ra^{-2-}$. 

\end{proof}

We now seek to close the argument by controlling the intermediate energies.  We let $\psi_c(z)$ be a smooth cut-off function to a compact subset of the spectrum bounded away from the threshold.
\begin{prop}\label{prop:int}

	If $|V(x)|\les \la x\ra^{-\delta}$, for some $\delta>3$,  we have the bound
	\begin{align*}
	\bigg|\int_{\R} e^{-it\sqrt{z^2+m^2}} \frac{z\psi_c(z)(z)}{\sqrt{z^2+m^2}} [\mR_V(z)](x,y)\, dz\bigg| \leq C_0 \min(1, |t|^{-\f12} ),
	\end{align*}
	where the constant $C_0$ depends on the support of the cut-off $\psi_c(z)$.
	Furthermore, if $\delta>5$ we may bound the above integral by 
	$C_0\min(1, |t|^{-\f12},|t|^{-\frac32}  \max(\la x\ra, \la y\ra )   )$.
	
\end{prop}
When $|z|$ is bounded both above and below, we employ the standard resolvent identity twice to write
$$
\mR_V(z)=\mR_0(z)-\mR_0(z)V\mR_0(z)+\mR_0(z)V\mR_V(z)V\mR_0(z).
$$
Inserting this into the Stone's formula, \eqref{eqn:stone}, requires us to bound three integrals.  The first integral, with only $\mR_0(z)$ is bounded by Theorem~\ref{thm:giggles} noting that $2^j \les 1$ in this case.  

\begin{lemma}\label{lem:hi single born}
	
	If $V\in L^{1,1}$, we have the bound
	\begin{align*}
	\bigg|\int_{\R} e^{-it\sqrt{z^2+m^2}} \frac{z\psi_c(z)}{\sqrt{z^2+m^2}} [\mR_0(z)V\mR_0(z)](x,y)\, dz\bigg| \les \min(1, |t|^{-\f12} ,|t|^{-\frac32} \max(\la x\ra, \la y\ra )   )
	\end{align*}
	
\end{lemma}

\begin{proof}
	
	We need to control the contribution of the iterated integral
	$$
	\bigg|\int_{\R^2} e^{-it\sqrt{z^2+m^2}} \frac{z\psi_c(z)}{\sqrt{z^2+m^2}} \mR_0(z)(x,x_1)V(x_1)\mR_0(z)(x_1,y)\, dx_1\,  dz\bigg|.
	$$
	A priori, using \eqref{resolventex} the integral in $x_1$ is bounded assuming that $V\in L^1$.  Using Fubini, we may integrate in $z$ first and let 
	$$
	\psi (z,x,y)= \frac{z \psi_c(z)(z)}{\sqrt{z^2+m^2}} \big[- \alpha     \,\textrm{sgn}(x-x_1) + \frac{1  }{ z}(  m \beta +\sqrt{z^2+m^2}I)\big] V(x_1) \big[- \alpha     \,\textrm{sgn}(x_1-y) + \frac{1  }{ z}(  m \beta +\sqrt{z^2+m^2}I)\big].
	$$
	Note that $\|\psi \|_{L^1}\les  |V(x_1)|$ and $\|\partial_z\psi \|_{L^1}\les |V(x_1)|$ uniformly in $x,y$. Therefore, using Lemma~\ref{lem:vdc}, and the fact that $V\in L^1$,  we can bound  by 
	$ \min(1, |t|^{-\f12} )$.  
	
	To obtain the weighted bound, we again apply Lemma~\ref{lem:vdc} (with $r=|x-x_1|+|x_1-y|$) noting that $j$ is bounded above,
	\begin{multline*}
	\big\| [\partial_{zz}+ir\partial_z]\big(\frac{\psi }z\sqrt{z^2+m^2}\big)\big\|_{L^1}\\
	=\big\| [\partial_{zz}+ir\partial_z]\big[- \alpha     \,\textrm{sgn}(x-x_1) + \frac{1  }{ z}(  m \beta +\sqrt{z^2+m^2}I)\big] V(x_1) \big[- \alpha     \,\textrm{sgn}(x_1-y) + \frac{1  }{ z}(  m \beta +\sqrt{z^2+m^2}I)\big]\psi_c(z)\big\|_{L^1}\\
	\les \la r\ra |V(x_1)|.
	\end{multline*}
	Assuming $V\in L^{1,1}$, we may bound the resulting $x_1$ integral by $\max(1,|x|,|y|)$ to obtain the desired result.
	
\end{proof}

\begin{lemma}\label{lem:RV int}
	
	If $|V(x)|\les \la x \ra^{-\delta}$ for some $\delta>3$, we have the bound
	\begin{align*}
		\bigg|\int_{\R} e^{-it\sqrt{z^2+m^2}} \frac{z\psi_c(z)(z)}{\sqrt{z^2+m^2}} [\mR_0(z)V\mR_V(z)V\mR_0(z)](x,y)\, dz\bigg| \les \min(1, |t|^{-\f12}).
	\end{align*}
	Furthermore, if $\delta>5$ we have
	\begin{align*}
		\bigg|\int_{\R} e^{-it\sqrt{z^2+m^2}} \frac{z\psi_c(z)(z)}{\sqrt{z^2+m^2}} [\mR_0(z)V\mR_V(z)V\mR_0(z)](x,y)\, dz\bigg| \les \min(1, |t|^{-\f12} ,|t|^{-\frac32}  \max(\la x\ra, \la y\ra )   ).
	\end{align*}
	
\end{lemma}

\begin{proof}
	
	We need to control the contribution of the iterated integral
	$$
	\bigg|\int_{\R^3} e^{-it\sqrt{z^2+m^2}} \frac{z\psi_c(z)}{\sqrt{z^2+m^2}} \mR_0(z)(x,x_1)V(x_1)\mR_V(z)(x_1,y_1) V(y_1) \mR_0(z)(y_1,y)\, dx_1\, dy_1\,  dz\bigg|.
	$$
	A priori, the integral in the spatial variables is bounded assuming that $V\in L^{2,\f12+}$.  Using Fubini, we may integrate in $z$ first and let 
	$$
	\psi (z,x,y)= \frac{z \psi_c(z)}{\sqrt{z^2+m^2}} e^{-iz|x-x_1|}\mR_0(z)(x,x_1) V(x_1)  \mR_V(z)(x_1,y_1) V(y_1) e^{-iz|y_1-y|}\mR_0(z)(y_1,y)
	$$
	Note that $\|\psi\|_{L^1}\les |V(x_1)\mR_V(z)(x_1,y_1) V(y_1)|$ and 
	$$
	\|\partial_z\psi\|_{L^1}\les |V(x_1)\mR_V(z)(x_1,y_1) V(y_1)|+ |V(x_1) \partial_z\mR_V(z)(x_1,y_1) V(y_1)|
	$$ 
	uniformly in $x,y$. Therefore, using Lemma~\ref{lem:vdc}, Lemma~\ref{lem:LAP} and the facts that $V\in L^{2,\f32+}$ and $j$ is bounded above,  we can bound by 
	$ \min(1, |t|^{-\f12})$.  
	
	To obtain the weighted bound, we again apply Lemma~\ref{lem:vdc} (with $r=|x-x_1|+|x_1-y|$) noting that
	\begin{multline*}
	\big\| [\partial_{zz}+ir\partial_z]\big(\frac{\psi }z\sqrt{z^2+m^2}\big)\big\|_{L^1}\\
	=\big\| [\partial_{zz}+ir\partial_z]\big[e^{-iz|x-x_1|}\mR_0(z)(x,x_1) V(x_1)  \mR_V(z)(x_1,y_1) V(y_1) e^{-iz|y_1-y|}\mR_0(z)(y_1,y)\big]\psi_c(z)\big\|_{L^1}\\
	\les   \la r\ra \sum_{\ell=0}^1|V(x_1) \partial_z^\ell \mR_V(z)(x_1,y_1) V(y_1)|+ |V(x_1) \partial_{zz} \mR_V(z)(x_1,y_1) V(y_1)|.
	\end{multline*}
	Noting that we need $|V(x)|\les \la x\ra^{-5-}$ to use Lemma~\ref{lem:LAP} on the second derivative of the resolvent, we may bound the resulting spatial integral by $\max(1,|x|,|y|)$.
	
\end{proof}

\end{document}